\documentclass[11pt]{amsart}
\usepackage{lmodern}
\usepackage{amsmath, amsthm, amssymb, amsfonts}
\usepackage[normalem]{ulem}
\usepackage{hyperref}

\usepackage{mathrsfs}

\usepackage{verbatim} 
\usepackage{longtable}
\usepackage{mathabx}

\usepackage{mathtools}

\usepackage{tikz}
\usetikzlibrary{decorations.pathmorphing}
\tikzset{snake it/.style={decorate, decoration=snake}}

\usepackage{caption}

\usepackage{tikz-cd}
\usetikzlibrary{arrows}

\theoremstyle{plain}
\newtheorem{thm}{Theorem}[section]

\newtheorem{lem}[thm]{Lemma}
\newtheorem{prop}[thm]{Proposition}
\newtheorem{conj}[thm]{Conjecture}
\newtheorem{question}[thm]{Question}

\theoremstyle{definition}

\theoremstyle{remark}
\newtheorem{rmk}[thm]{Remark}

\newcommand{\BA}{{\mathbb{A}}}
\newcommand{\BB}{{\mathbb{B}}}
\newcommand{\BC}{{\mathbb{C}}}

\newcommand{\BG}{{\mathbb{G}}}

\newcommand{\BK}{{\mathbb{K}}}

\newcommand{\BP}{{\mathbb{P}}}
\newcommand{\BQ}{{\mathbb{Q}}}
\newcommand{\BR}{{\mathbb{R}}}

\newcommand{\BX}{{\mathbb{X}}}

\newcommand{\BZ}{{\mathbb{Z}}}

\newcommand{\CA}{{\mathcal A}}
\newcommand{\CB}{{\mathcal B}}

\newcommand{\CE}{{\mathcal E}}
\newcommand{\CF}{{\mathcal F}}

\newcommand{\CH}{{\mathcal H}}

\newcommand{\CL}{{\mathcal L}}

\newcommand{\CO}{{\mathcal O}}

\newcommand{\CR}{{\mathcal R}}

\newcommand{\CU}{{\mathcal U}}

\DeclareFontFamily{OT1}{rsfs}{}
\DeclareFontShape{OT1}{rsfs}{n}{it}{<-> rsfs10}{}
\DeclareMathAlphabet{\curly}{OT1}{rsfs}{n}{it}

\newcommand{\git}{\mathbin{
  \mathchoice{/\mkern-6mu/}
    {/\mkern-6mu/}
    {/\mkern-5mu/}
    {/\mkern-5mu/}}}

\usepackage{tikz}
\usepackage{lmodern}
\usetikzlibrary{decorations.pathmorphing}

\usepackage[all,cmtip]{xy}

\makeatletter
\let\@wraptoccontribs\wraptoccontribs
\makeatother

\begin{document}
\title[Topology of Galois conjugate character varieties]{Topology of Galois conjugate character varieties}
\date{\today}

\author[J. Shen]{Junliang Shen}
\address{Yale University}
\email{junliang.shen@yale.edu}

\author[S. Zhang]{Siqing Zhang}
\address{Yale University}
\email{siqing.zhang@yale.edu}


\begin{abstract}
We study the interaction between integral structures, automorphisms, and tautological relations for the cohomology of character varieties. Based on this, we propose a method to detect differences in the homotopy types of Galois conjugate character varieties. As an application, we find the first example of a pair of Galois conjugate character varieties that are not homotopy equivalent, answering negatively a 2005 question of Hausel.

\end{abstract}

\maketitle

\setcounter{tocdepth}{1} 

\tableofcontents
\setcounter{section}{-1}

\section{Introduction}
Throughout, we work over the complex numbers $\BC$.

\subsection{Galois conjugate character varieties}\label{sec0.1}

Let $C$ be a nonsingular projective connected curve of genus $g\geq 2$, and let $n\geq 2$ be a positive integer. We set $\xi_n:= e^{\frac{2\pi \sqrt{-1}}{n}}$ to be the standard primitive $n$-th root of unity. For any $d$ coprime to $n$, the character variety associated with $C,n,d$ is the affine GIT quotient with respect to the conjugation action:
\begin{equation*} M_{n,d}^B := \Big{\{}a_k, b_k \in \mathrm{GL}_n,~k=1,2,\dots,g: ~~\prod_{j=1}^g [a_j, b_j] = \xi_n^d\mathrm{Id}_n \Big{\}}\git \mathrm{GL}_n. \end{equation*}
This is also known as the \emph{Betti moduli space}, which is diffeomorphic to the \emph{Dolbeault moduli space} $M^\mathrm{Dol}_{n,d}$ of stable Higgs bundles on $C$ of rank $n$ and degree $d$ by non-abelian Hodge theory \cite{Hit, Hit1, Simp, Si1994II}.

The classical Narasimhan--Seshadri correspondence \cite{NS} gives a diffeomorphism between the unitary moduli space
\begin{equation*} N_{n,d}^B := \Big{\{}a_k, b_k \in \mathrm{U}_n,~k=1,2,\dots,g: ~~\prod_{j=1}^g [a_j, b_j] = \xi_n^d\mathrm{Id}_n \Big{\}}\git \mathrm{U}_n \subset M^B_{n,d} \end{equation*}
and the moduli space $N^{\mathrm{Dol}}_{n,d}$ of stable vector bundles on $C$ of rank $n$ and degree $d$. The non-abelian Hodge theory extends this to a diffeomorphism 
\[
M^{\mathrm{Dol}}_{n,d} \xrightarrow[\simeq]{~~C^\infty~~} M^B_{n,d},
\]
where $N^{\mathrm{Dol}}_{n,d}$ naturally embeds in $M^{\mathrm{Dol}}_{n,d}$ as the locus characterized by the vanishing of the Higgs field. When we are only concerned with the underlying manifolds or topological spaces, we use $M_{n,d}$ and $N_{n,d}$ to denote these spaces for convenience.

Over 50 years ago, Harder--Narasimhan \cite{HN} proved that the manifolds $N_{n,d}, N_{n,d'}$ are homeomorphic if and only if
\begin{equation}\label{condition}
d \equiv \pm d' \pmod{n},
\end{equation}
where Betti numbers were used to distinguish the topology of $N_{n,d}, N_{n,d'}$. Under the condition (\ref{condition}), there is an isomorphism of varieties $N^{\mathrm{Dol}}_{n,d} \simeq N^{\mathrm{Dol}}_{n,d'}$.

The question of the topological dependence of $M_{n,d}$ on $d$ is more subtle and more interesting. In \cite{H_Survey, HRV}, it was observed that $M^B_{n,d}, M^B_{n,d'}$ are Galois conjugate varieties --- they are given by the underlying topological spaces of the complex varieties obtained from the same abstract variety defined over $\BQ(\xi_n)$ via different embeddings $\BQ(\xi_n)\hookrightarrow \BC$. Therefore, they have identical Betti numbers. Furthermore, most Galois conjugate varieties are homeomorphic over $\BC$, \emph{e.g.} \cite{Reed}. The following question of Hausel from 2005 asks whether there is a homeomorphism connecting Galois conjugate character varieties.

\begin{question}\cite[Problem 3.12]{H_Survey}\label{question}
For $d,d'$ coprime to $n$, are $M_{n,d}$ and $M_{n,d'}$ homeomorphic?
\end{question}

A homeomorphism between $M_{n,d}$ and $M_{n,d'}$, if it were to exist, would be very mysterious, as it could not preserve the unitary locus by the Harder–Narasimhan result. On the other hand, to the best of our knowledge, all topological invariants computed so far do not distinguish $M_{n,d}$ from $M_{n,d'}$. The purpose of this paper is to explore new cohomological structures for character varieties which are strong enough to distinguish the topology of Galois conjugate character varieties. 

We note that when $n\leq 4$, any pair $d,d'$ coprime to $n$ must satisfy (\ref{condition}). The corresponding manifolds $M_{n,d}, M_{n,d'}$ are obviously homeomorphic, as their Dolbeault models are isomorphic $M^{\mathrm{Dol}}_{n,d} \simeq M^{\mathrm{Dol}}_{n,d'}$ as algebraic varieties. The first nontrivial example to study Question \ref{question} occurs when $(g,n)=(2,5)$, for which our method yields a negative answer.

\begin{thm}\label{thm0.2}
    Let $C$ be a curve of genus $g=2$. The topological spaces $M_{5,1}$ and $M_{5,2}$ associated with $C$ are not homotopy equivalent.
\end{thm}

In fact, our most structural results apply to arbitrary character varieties; see Theorems \ref{thm0.5}, Proposition \ref{normal_group}, and Theorem \ref{thm0.8}. The reason for focusing on the case $(g,n)= (2,5)$ is computational complexity: a key step in producing counterexamples requires solving a large linear algebra problem in the proof of Proposition \ref{prop3.12}. Stronger computational resources may yield further counterexamples; see Remark \ref{rmk_more}.

\subsection{Cohomology}\label{sec0.2}

There has been a long history of finding topologically distinct Galois conjugate varieties starting from Serre \cite{Serre} who used the fundamental group. However, for character varieties with $n> 2$ or $g>2$, we have
\[
\pi_1(M_{n,d}) \simeq \BZ^{2g}
\]
by \cite[Theorems 4.1 and 4.2]{BGPG}, which is clearly independent of $d$. Furthermore, we show in Section \ref{Homotopy} that all the higher homotopy groups are $d$-independent,
\[
\pi_k(M_{n,d}) \simeq \pi_k(M_{n,d'}), \quad k> 1, \quad (n,d)=(n,d')=1.
\]
This means homotopy groups are not powerful enough for our purpose.

We switch to cohomological invariants. The fact that $M^B_{n,d}, M^B_{n,d'}$ are Galois conjugate implies that they have isomorphic $\ell$-adic cohomology rings. In fact, it was further known that Galois conjugate character varieties share isomorphic rational cohomology rings, in the strongest sense.\footnote{We note that in general Galois conjugate varieties may not have isomorphic rational (or even real) cohomology rings; see for example \cite{Charles}.} The following theorem was obtained as a special case of \cite[Theorem 1.1(b)]{dCMSZ} using the Dolbeault moduli spaces, which can also be deduced from the observation of \cite[Remark 4.8]{Survey} via the Betti moduli spaces.

\begin{thm}\label{thm0.3}
Assume that the curve $C$ has genus $g\geq 2$, and assume that $(n,d) = 1$. We have an isomorphism of graded $\BQ$-algebras
\begin{equation}\label{coh_thm0.3}
H^*(M_{n,d}, \BQ)  \simeq H^*((\BC^*)^{2g},\BQ) \otimes {R}_{g,n}.
\end{equation}
Here $R_{g,n} = \oplus_{i} R_{g,n}^i$ is a graded-commutative algebra with tautological generators 
\begin{equation}\label{generators}
\alpha_i \in  R_{g,n}^{2i-2},~~ \beta_i \in R_{g,n}^{2i}, ~~ \psi_{i,j} \in R_{g,n}^{2i-1}, \quad 2\leq i \leq n, \quad 1\leq j \leq {2g}
\end{equation}
whose ideal of relations $I_{g,n}$ among these generators only depends on $n,g$, and is homogeneous with respect to the grading
\begin{equation}\label{Chern_grading}
\mathrm{deg}^C(\alpha_i) = \mathrm{deg}^C(\beta_i) = \mathrm{deg}^C(\psi_{i,j}) = i.
\end{equation}
In particular, the $\BQ$-algebra $R_{g,n}$ is essentially bigraded and is $d$-independent.
\end{thm}

We will review Theorem \ref{thm0.3} in Section \ref{pf_of_thm_0.3}, where $R_{g,n}$ can be realized as the cohomology of the moduli of $\mathrm{PGL}_n$-character variety. Although it seems that this theorem only gives $d$-independent information, it turns out to be a key ingredient in the proof of Theorem \ref{thm0.2}. The grading (\ref{Chern_grading}) is the \emph{Chern grading}. It can be regarded either as a grading splitting the weight filtration for $M^B_{n,d}$ or as a grading splitting the perverse filtration for the Hitchin system associated with $M^{\mathrm{Dol}}_{n,d}$, in the context of the (now-proven) $P=W$ conjecture \cite{dCHM1,Shende,MS_PW,HMMS}.


We conjecture that the topological types of Galois conjugate character varieties are all distinct unless the Dolbeault moduli spaces are obviously isomorphic (\emph{i.e.}, the condition (\ref{condition}) holds); furthermore, this is detected by the integral cohomology ring.

\begin{conj}\label{main_conj}
    Assume that $(n,d)= (n,d')=1$. We have an isomorphism of integral cohomology rings
    \[
    \left( H^*(M_{n,d}, \BZ), ~~\cup \right) \simeq  \left( H^*(M_{n,d'}, \BZ), ~~\cup \right)
    \]
    as graded algebras if and only if (\ref{condition}) holds.
\end{conj}

Our main evidence for Conjecture \ref{main_conj} is Theorem \ref{thm0.5} below, which relies on cohomological automorphisms.


\subsection{Cohomological automorphisms}
Let $G_{g,n}$ be the group of graded automorphisms of the graded $\BQ$-algebra $R_{g,n} = \oplus_i R_{g,n}^i$ in Theorem \ref{thm0.3}, \emph{i.e.},
\[
G_{g,n} = \mathrm{Aut}^{\mathrm{gr}}(R_{g,n}).
\]
Elements in $G_{g,n}$ are not assumed to preserve the Chern grading (\ref{Chern_grading}) of Theorem \ref{thm0.3}. We consider the low even degree components of $R_{g,n}$:
\[
R_{g,n}^2 = \BQ \alpha_2, \quad R_{g,n}^4= \BQ \alpha_2^2 \oplus \BQ\alpha_3 \oplus \BQ \beta_2.
\]
Clearly, any $\varphi \in G_{g,n}$ preserves the line $\BQ\alpha_2$, and therefore the line $\BQ \alpha^2_2$ in $R^4_{g,n}$. We say that $\varphi \in G_{g,n}$ is \emph{weakly diagonal} if $\varphi$ further preserves both lines $\BQ\alpha_3$ and $\BQ\beta_2$ in $R^4_{g,n}$, \emph{i.e.}
\[
\varphi(\alpha_3) = \mu \alpha_3, \quad \varphi(\beta_2)= \nu \beta_2, \quad \quad \mu, \nu \in \BQ^\ast.
\]
Weakly diagonal elements form a subgroup $G_{g,n}^{\mathrm{wd}} \subset G_{g,n}$.

\begin{thm}\label{thm0.5}
If every element in $G_{g,n}$ is weakly diagonal, \emph{i.e.}, 
\begin{equation}\label{wd=G}
G^{\mathrm{wd}}_{g,n} = G_{g,n},
\end{equation}
then Conjecture \ref{main_conj} holds for $g$ and $n$.
\end{thm}

The theorem reduces the question of distinguishing the homotopy types of Galois conjugate character varieties to the study of the $d$-independent group $G_{g,n}$. In Section \ref{sec4}, we provide a linear-algebraic criterion to verify (\ref{wd=G}) for any pair $(g,n)$ (see Proposition \ref{normal_group}), and use this method to prove the following theorem which deduces Theorem \ref{thm0.2} from Theorem \ref{thm0.5}.

\begin{thm}\label{thm0.7}
We have $G^{\mathrm{wd}}_{2,5} = G_{2,5}$.
\end{thm}

A key geometric input is a distinguished tautological relation we found for the character variety associated with $(g,n)=(2,5)$ which we discuss in the next section.

\subsection{Tautological relations}

The group $G_{g,n}$ is governed by the relations between the tautological generators (\ref{generators}) --- the action of any element $\varphi \in G_{g,n}$ on (\ref{generators}) has to preserve the ideal of relations. Realizing this approach has two challenges. First, finding relations among the tautological classes is a longstanding question for $M_{n,d}$. When $n=2$, this was resolved by Hausel--Thaddeus \cite{HT2} more than 2 decades ago; since then not much progress was made in higher ranks. But in order to attack Question \ref{question}, we have to consider the cases of $n\geq 5$. Second, the ideal of relations in general is quite huge and it is difficult to compute the stabilizer. 

Our proof of Theorem \ref{thm0.7} relies on the general structural theorem (Theorem \ref{thm0.8}) below. It is based on the fact that there is a \emph{distinguished tautological relation} for rank $n\geq 3$. The proof of Theorem \ref{thm0.8} combines the following ingredients:
\begin{enumerate}
    \item[(i)] The closed formulas for the Poincar\'e polynomials for $M_{n,d}$  conjectured by Hausel--Rodriguez-Villegas \cite{HRV} and proved by Schiffmann \cite{Sch} and Mellit \cite{Mellit}.
    \item[(ii)] The tautological relations for the more classical moduli spaces $N_{n,d}$ described by Earl--Jeffrey--Kirwan \cite{EK, JK}.
\end{enumerate}
More precisely, (i) is used to show that there is only one relation in the lowest possible cohomological degree; (ii) combined with Theorem \ref{thm0.3} is applied to find candidates of the relations for $M_{n,d}$ from the Earl--Jeffrey--Kirwan relations for $N_{n,d}$; we call these candidates \emph{homogeneous} EJK relations. Miraculously, this process completely determines the desired distinguished relation for $g=2$ and $n=5$.

\begin{thm}\label{thm0.8} Assume $g \geq 2, n\geq 3$, and we set 
\[
r_{g,n} := 4g(n-1)-2n.
\]
The following statements hold for the graded $\BQ$-algebra $R_{g,n}$.
\begin{enumerate}
    \item[(a)] There are no relations among the classes (\ref{generators}) in degrees $< r_{g,n}$.
    \item[(b)] There is a unique relation 
    \[
    \Upsilon_{g,n} = 0 
    \]
    in $R^{r_{g,n}}_{g,n}$.
    \item[(c)] When $g=2$ and $n= 5$, there is a unique homogeneous symplectic-invariant EJK relation in degree $r_{2,5}=22$ that recovers the relation $\Upsilon_{2,5}=0$.
\end{enumerate}

\end{thm}

In summary, the idea of the proof of Theorem \ref{thm0.2} lies in the following observations: Theorem \ref{thm0.5} shows for general $d,d'$ that any isomorphism
\begin{equation}\label{isom}
H^*(M_{n,d}, \BQ) \xrightarrow{\simeq} H^*(M_{n,d'}, \BQ)
\end{equation}
of graded $\BQ$-algebras preserving the integral structures cannot behave \emph{too nicely} on the tautological generators (\ref{generators}). The formula given by Theorem \ref{thm0.8}(c), on the other hand, shows that the distinguished relation forces that the isomorphism (\ref{isom}) has to act nicely on the tautological generators for $(g,n)=(2,5)$.

\begin{rmk}
    Theorem \ref{thm0.3} is used in the proof of Theorem \ref{thm0.2} twice independently. First, it is applied to identify canonically the rational cohomology for the spaces associated with various $d$; this allows us to reduce the study of integral cohomology to the study of the automorphism group $G_{g,n}$ as in Theorem \ref{thm0.5}. Second, the Chern grading is a crucial input in the proofs of Theorem \ref{thm0.7} and Theorem \ref{thm0.8}(c).
\end{rmk}

\subsection{Relations to other work}


As discussed in Section \ref{sec0.1}, Betti numbers were used by Harder--Narasimhan \cite{HN} in the 1970s to prove topological degree-dependence for moduli spaces of stable vector bundles on a curve. More recently, an interesting class of moduli spaces analogous to moduli spaces of stable vector bundles on curves has been studied: the moduli spaces of stable 1-dimensional sheaves on $\BP^2$. Although these spaces have identical Betti numbers as the Euler characteristic of the sheaves varies \cite{PB, MS_chi}, they have distinct rational cohomology rings \cite{LMP}.

The topological degree-dependence for character varieties is more subtle, since these varieties are Galois conjugate and even have isomorphic rational cohomology rings. Therefore, their topological differences, if they exist, are hidden deeper. To the best of our knowledge, the integral cohomology ring is the first known topological invariant that is sensitive to the degree of the character variety.

On the other hand, over the decades there have been many results concerning degree-independence properties for Dolbeault or Betti moduli spaces; see \cite{dCZ, dCMSZ, Da, GWZ, HRV, KK, MS_Pi, MSY2, Mellit, Sch, Yu}.

In the appendix, a short proof of the degree-dependence of the Dolbeault moduli space as a complex analytic space is provided.

\subsection{Acknowledgements}
We would like to thank David Zhiyuan Bai, David Fang, Davesh Maulik, Miguel Moreira, Weite Pi, Woonam Lim, and Qizheng Yin for helpful discussions on character varieties and tautological relations over the years. The result presented in Appendix was obtained during discussions with Andres Fernandez Herrero, and we are grateful to him for his helpful input. Our interest in studying the topology of Galois conjugate character varieties comes from Hausel’s beautiful survey articles \cite{H_Survey, Survey}.

The paper was completed while J.S. was visiting Universit\"at Bonn in July 2026.
J.S.~was supported by the NSF grant DMS-2301474 and a Sloan Research Fellowship. S.Z.~was supported by an AMS-Simons Travel Grant.

\section{Moduli spaces and tautological classes}\label{sec1}

In this section, we fix $C$ to be a curve of genus $g\geq 2$, and we fix a standard symplectic basis of $H^1(C, \BZ)$,
\[
\delta_j \in H^1(C, \BZ), \quad j=1,2,\cdots, 2g
\]
with intersection matrix
\[
\left( \int_C \delta_j \cup \delta_k\right)_{1\leq j,k\leq 2g}=
\begin{pmatrix}
0 & \mathrm{Id}_g \\
-\mathrm{Id}_g & 0
\end{pmatrix}.
\]
Assume $n,d$ are coprime integers with $n\geq 2$.

\subsection{Moduli spaces}\label{sec1.1}
The graded algebra $R_{g,n}$ is given by the cohomology of the moduli space of $\mathrm{PGL}_n$-character varieties which we review as follows. The references are \cite{HT1} and \cite[Section 2.4]{dCHM1}.

It is more convenient to first describe the moduli spaces associated with $\mathrm{SL}_n$. For degree $d$ coprime to $n$, the Betti moduli space associated with $\mathrm{SL}_n$ and $d$ is given by the GIT quotient
\begin{equation*} \widecheck{M}_{n,d}^B := \Big{\{}a_k, b_k \in \mathrm{SL}_n,~k=1,2,\dots,g: ~~\prod_{j=1}^g [a_j, b_j] = \xi_n^d\mathrm{Id}_n \Big{\}}\git \mathrm{PGL}_n. \end{equation*}
The corresponding Dolbeault moduli space $\widecheck{M}^{\mathrm{Dol}}_{n,d}$ parameterizes stable Higgs bundles $(\CE, \theta)$ on the curve $C$ with
\[
\mathrm{rk}(\CE) = n, \quad \mathrm{det}(\CE) \simeq L , \quad \mathrm{tr}(\theta) = 0,
\]
where $L \in \mathrm{Pic}^d(C)$ is a fixed line bundle on $C$ of degree $d$. It is clear that the Dolbeault moduli spaces are isomorphic for different choices of line bundles in $\mathrm{Pic}^d(C)$.

Recall that the \emph{abelian} moduli spaces (associated with rank $n=1$ and degree $d=0$) 
\[
 M^B_{1,0} = (\BC^*)^{2g},  \quad M^{\mathrm{Dol}}_{1,0} = \mathrm{Pic}^0(C) \times H^0(C, \omega_C)
\]
are naturally groups admitting actions on $M^B_{n,d}$ and $M^{\mathrm{Dol}}_{n,d}$ respectively. This induces natural actions of the finite group 
\[
\Gamma = \left(\mathrm{\BZ}/n\BZ \right)^{\oplus 2g}
\]
on the moduli spaces $\widecheck{M}^B_{n,d}$ and $\widecheck{M}^{\mathrm{Dol}}_{n,d}$, where the group $\Gamma$ is viewed as the subgroup of $n$-torsion elements in $M^{B}_{1,0}$ or $M^{\mathrm{Dol}}_{1,0}$.

The Betti and the Dolbeault moduli spaces associated with $\mathrm{PGL}_n$ and $d$ are the quotient spaces:
\[
\widehat{M}^{{B}}_{n,d}:= \widecheck{M}^{B}_{n,d}/\Gamma, \quad \widehat{M}^{\mathrm{Dol}}_{n,d}:= \widecheck{M}^{\mathrm{Dol}}_{n,d}/\Gamma.
\]
The non-abelian Hodge theory yields natural diffeomorphisms between the Betti and the Dolbeault moduli spaces
\[
M^\mathrm{Dol}_{n,d} \xrightarrow[~~C^\infty~~]{\simeq} M^B_{n,d}, \quad \widecheck{M}^\mathrm{Dol}_{n,d} \xrightarrow[~~C^\infty~~]{\simeq} \widecheck{M}^B_{n,d}, \quad \widehat{M}^\mathrm{Dol}_{n,d} \xrightarrow[~~C^\infty~~]{\simeq} \widehat{M}^B_{n,d}.
\]
As before, when we are only concerned with the underlying topology of the moduli spaces, we use $\widecheck{M}_{n,d}, \widehat{M}_{n,d}$ to denote the spaces associated with $\mathrm{SL}_n$ and $\mathrm{PGL}_n$ respectively.

The cohomology of the $\mathrm{GL}_n$ space $M_{n,d}$ can be recovered by the $\mathrm{PGL}_n$ space in a simple form. More precisely, there is a natural diagonal action of $\Gamma$ on the product $M_{1,0} \times \widecheck{M}_{n,d}$ whose resulting quotient is 
\[
M_{n,d} = ( M_{1,0} \times \widecheck{M}_{n,d})\big{/}\Gamma.
\]
This, combined with the K\"unneth decomposition, yields an isomorphism of cohomology as graded $\BQ$-algebras
\begin{equation}\label{coh_GL/PGL}
    H^*(M_{n,d}, \BQ) = H^*(M_{1,0}, \BQ) \otimes H^*(\widehat{M}_{n,d}, \BQ).
\end{equation}
We note that (\ref{coh_GL/PGL}) is exactly the isomorphism (\ref{coh_thm0.3}) in Theorem \ref{thm0.3}. It remains to justify that the cohomology $H^*(\widehat{M}_{n,d}, \BQ)$ is bigraded and $d$-independent. Both properties rely on the tautological generators.

\subsection{Tautological classes}\label{sec1.2}

There are two ways of constructing tautological classes, via either the Dolbeault moduli space \cite{Markman} or the Betti moduli space \cite{Shende}. We need both for our purpose. The match of these two constructions via the non-abelian Hodge theory was explained in \cite{HT1}.

We first review the construction via the Dolbeault moduli space. Let $p_C$ and $p_M$ be the projections from $C\times M_{n,d}$ to $C$ and $M_{n,d}$ respectively.
We view $M_{n,d}$ to be the Dolbeault moduli space, and let $\CU_d$ be a universal bundle over $C\times M^{\mathrm{Dol}}_{n,d}$ whose existence is guaranteed by $(n,d)=1$. The bundle $\CU_d$ is not unique. For each class 
\[
\alpha=p_C^*\alpha_C+p_M^*\alpha_M\in H^2(C\times M_{n,d},\BQ),
\]
we consider the twisted Chern character
\[\mathrm{ch}^{\alpha}(\mathcal{U}_d):=\mathrm{ch}(\mathcal{U}_d)\cup \mathrm{exp}(\alpha) \in H^*(C\times M_{n,d}, \BQ).\]
We choose $\alpha$ such that,
given any points $c_0\in C$ and $m_0 \in M_{n,d}$, we have that 
\begin{equation}\label{eqn: alpha}
    \mathrm{ch}_1^{\alpha}(\CU_d)|_{c_0\times M_{n,d}}=0\in H^2(M_{n,d},\BQ), \text{\quad and\quad } \mathrm{ch}^{\alpha}_1(\CU_d)|_{C\times m_0}=0\in H^2(C,\BQ).
\end{equation}
Equivalently, via K\"unneth decomposition
\[
\mathrm{ch}_1^{\alpha}(\CU_d) \in H^1(C, \BQ) \otimes H^1(M_{n,d}, \BQ).
\]
Define the tautological classes
\[c_i(\gamma):=\int_{\gamma}\mathrm{ch}^{\alpha}_i(\CU_d)=p_{M*}(p_C^*\gamma\cup \mathrm{ch}_i^{\alpha}(\CU_d)), \quad \gamma \in H^*(C, \BQ).
\]
The classes $c_i(\gamma)$ are independent of the choice of $\CU_d$.


\begin{thm}[\cite{Markman}]\label{thm1.1}
    The tautological classes
    \[
    c_i(\mathbf{1}_C) \in H^{2i-2}(M_{n,d}, \BQ), \quad c_i(\delta_j)\in H^{2i-1}(M_{n,d}, \BQ), \quad c_i(\mathrm{pt}_C) \in  H^{2i}(M_{n,d}, \BQ)
    \]
    with 
    \[
    i=1,\cdots, n, \quad j=1,2,\cdots, 2g
    \]
    generate $H^*(M_{n,d}, \BQ)$ as a graded $\BQ$-algebra.
\end{thm}
    
\begin{proof}
Markman \cite{Markman} proved that $c_i(\gamma)$ with all $i\in \BZ_{\geq 1}$ and $\gamma \in H^*(C, \BQ)$ generate $H^*(M_{n,d}, \BQ)$ as a graded $\BQ$-algebra. Now since $\CU_d$ is a vector bundle of rank $n$, any class $\mathrm{ch}_k(\CU_d)$ can be expressed in terms of $\mathrm{ch}_i(\CU_d)$ with $i\leq n$. We conclude that the classes $c_i(\gamma)$ with $i\leq n$ are sufficient as generators. 
\end{proof}

We now present the tautological classes in terms of the isomorphism (\ref{coh_GL/PGL}). Due to the normalization condition, the classes 
\[
c_1(\delta_j), \quad j=1,2,\cdots, 2g
\]
span 
\[
H^1(M_{1,0}, \BQ) \otimes H^0(\widehat{M}_{n,d}, \BQ) = H^1(M_{n,d}, \BQ).
\]
The K\"unneth decomposition yields a natural projection 
\[
H^*(M_{n,d}, \BQ) \twoheadrightarrow H^0(M_{1,0}, \BQ) \otimes H^*(\widehat{M}_{n,d}, \BQ) = H^*(\widehat{M}_{n,d}, \BQ).
\]
We define the classes 
\begin{equation}\label{PGL_taut}
\alpha_i \in H^{2i-2}(\widehat{M}_{n,d}, \BQ), \quad \beta_i \in H^{2i}(\widehat{M}_{n,d}, \BQ), \quad \psi_{i,j} \in H^{2i-1}(\widehat{M}_{n,d}, \BQ), \quad 2\leq i\leq n, ~~1\leq j\leq 2g, 
\end{equation}
to be the images of the classes
\[
c_i(\mathbf{1}_C) \in H^{2i-2}(M_{n,d}, \BQ), \quad c_i(\mathrm{pt}_C) \in H^{2i}(M_{n,d}, \BQ), \quad c_i(\delta_j) \in H^{2i-1}(M_{n,d}, \BQ)
\]
under the projection. As a corollary of Theorem \ref{thm1.1}, we obtain that the classes (\ref{PGL_taut}) generate $H^*(\widehat{M}_{n,d}, \BQ)$ as a graded $\BQ$-algebra. We note that the classes (\ref{PGL_taut}) can also be constructed on $\widehat{M}_{n,d}$ directly using the universal projective bundle over $C\times \widehat{M}_{n,d}$; see \cite[Section 2.1]{MS_PW}. Here we consider the generators in the cohomology $H^*(M_{n,d}, \BQ)$ first because it is more natural to work with the integral structure there.

As we mentioned earlier, we can also use the Betti moduli space to construct the tautological classes (\ref{PGL_taut}); this approach was used in \cite{Shende} to calculate the weights of these classes with respect to the mixed Hodge structure on $\widehat{M}^B_{n,d}$.

Let $\Delta_C$ be the simplicial set with geometric realization homotopic to the curve $C$. It admits the same rational cohomology as $C$,
\[
H^*(\Delta_C, \BQ)  = H^*(C, \BQ),
\]
as graded $\BQ$-vector spaces, but the Hodge structure is trivial for $\Delta_C$. Shende observed that there is a natural evaluation morphism
\[
\mathrm{ev}: \Delta_C \times \widehat{M}^B_{n,d} \to B\mathrm{PGL}_n
\]
whose pullback 
\[
\mathrm{ev}^*: H^*(B\mathrm{PGL}_n, \BQ) \to H^*(\Delta_C, \BQ) \otimes H^*(\widehat{M}^B_{n,d}, \BQ)
\]
is a morphism of Hodge structures. Recall that the cohomology $H^*(B\mathrm{PGL}_n, 
\BQ)$ is generated by the tautological classes 
\[
\mathrm{ch}_i \in H^{2i}(B\mathrm{GL}_n, \BQ)
\]
given by the Chern characters of the universal principal bundle. The classes (\ref{PGL_taut}) are then characterized by
\[
 \mathrm{ev}^*\mathrm{ch}_i = \mathrm{pt}_C \otimes \alpha_i + \mathbf{1}_C \otimes \beta_i + \sum_{j=1}^{2g} \delta_j \otimes \psi_{i,j}.
\]

\begin{thm}[\cite{Shende}]\label{thm1.2}
    There is a unique multiplicative bigraded decomposition
    \[
    H^*(\widehat{M}_{n,d}, \BQ) = \bigoplus_{i,k} H^k_{(i)}(\widehat{M}_{n,d}, \BQ)
    \]
    satisfying that 
    \[
    \alpha_i \in H_{(i)}^{2i-2}(\widehat{M}_{n,d}, \BQ), \quad \beta_i \in H_{(i)}^{2i}(\widehat{M}_{n,d}, \BQ), \quad \psi_{i,j} \in H_{(i)}^{2i-1}(\widehat{M}_{n,d}, \BQ)
    \]
\end{thm}

\begin{proof}
    Shende \cite{Shende} proved that the cohomology of $\widehat{M}_{n,d}$ is of Hodge--Tate type; furthermore, the Hodge filtration splits the weight filtration and the decomposition given by
    \[
    H_{(i)}^*(\widehat{M}_{n,d}, \BQ) := \left(W_{2i} \cap F^i\right) H^*(\widehat{M}^B_{n,d}, \BQ)
    \]
    satisfies the desired properties. The uniqueness follows from the fact that the tautological classes (\ref{PGL_taut}) generate the cohomology.
\end{proof}

We define the \emph{Chern grading} to be the one given by the second grading of Theorem \ref{thm1.2}; under this grading, any K\"unneth component of $\mathrm{ev}^*\mathrm{ch}_i$ has degree $i$. The $P=W$ conjecture \cite{dCHM1, MS_PW, HMMS} was proven by showing that 
\[
P_k H^*(\widehat{M}_{n,d}, \BQ) = \bigoplus_{i\leq k} H_{(i)}^*(\widehat{M}_{n,d}, \BQ)
\]
where $P_\bullet$ is the perverse filtration associated with the Hitchin system given by $\widehat{M}^{\mathrm{Dol}}_{n,d}$. In other words, the Chern grading splits either the perverse filtration for $\widehat{M}^\mathrm{Dol}_{n,d}$ or the weight filtration for $\widehat{M}^B_{n,d}$.

\subsection{Proofs of Theorem \ref{thm0.3}}\label{pf_of_thm_0.3}

In view of Theorem \ref{thm1.2}, it suffices to show that, for $d,d'$ coprime to $n$, there is an isomorphism of graded $\BK$-algebras
\[
\varphi_{d,d'}: H^*(\widehat{M}_{n,d}, \BK) \xrightarrow{\simeq} H^*(\widehat{M}_{n,d'}, \BK) 
\]
preserving the tautological classes (\ref{PGL_taut}):
\[
\varphi_{d,d'}(\alpha_i)= \alpha_i, \quad \varphi_{d,d'}(\beta_i)= \beta_i, \quad \varphi_{d,d'}(\psi_{i,j})= \psi_{i,j}.
\]
Here $\BK$ can be any field which contains $\BQ$ as a subfield. Indeed, since all the tautological classes are rational, such an isomorphism would automatically preserve the rational cohomology rings.

This statement is an immediate consequence of \cite[Theorem 1.1(b)]{dCMSZ}. Furthermore, Theorem \ref{thm1.2} also follows from \cite[Theorem 1.1(a)]{dCMSZ} where the Chern grading is characterized by the weight decomposition associated with the action of the center $\BC^* \subset \mathrm{GSp}_{2g}(\BQ)$. The approach of \cite{dCMSZ} relies on the algebraic geometry of the Dolbeault moduli spaces $\widehat{M}^{\mathrm{Dol}}_{n,d}$ in positive characteristic \cite{OV, Groch, Chen-Zhu}.

In the following, for the reader's convenience, we provide another proof of Theorem \ref{thm0.3} which relies on the algebraic geometry of the Betti moduli spaces; the idea was sketched in \cite[Remark 4.8]{Survey} based on the fact that $\widehat{M}^B_{n,d}$ and $\widehat{M}^B_{n,d'}$ are Galois conjugate varieties. Since we need to use $\ell$-adic cohomology, we fix from now on a prime $\ell$.

Recall the standard $n$-th root of unity $\xi_n$. 
Let $\sigma \in \mathrm{Gal}(\overline{\BQ}/\BQ)$ be an element satisfying
\[
\sigma(\xi^d_n) = \xi_n^{d'}.
\]
By definition, the character varieties $\widehat{M}^B_{n,d}$ and $\widehat{M}^B_{n,d'}$ are Galois conjugate with respect to $\sigma$,
\[
 \widehat{M}^B_{n,d'} =  {^\sigma}{\widehat{M}^B_{n,d}}.
\]
Therefore, the element $\sigma$ induces an isomorphism of $\ell$-adic cohomology rings:
\begin{equation}\label{Gal-coh}
 \sigma:   H^*(\widehat{M}^B_{n,d}, \BQ_\ell) \xrightarrow{\simeq} H^*(\widehat{M}^B_{n,d'}, \BQ_\ell).
\end{equation}
The rational tautological classes (\ref{PGL_taut}) are natural classes in the $\ell$-adic cohomology via the canonical embedding $\BQ \subset \BQ_\ell$. We now analyze how (\ref{Gal-coh}) interacts with the tautological classes.

By the discussion before Theorem \ref{thm1.2}, the tautological classes (\ref{PGL_taut}) on the character variety are obtained by pulling back the tautological classes on $B\mathrm{PGL}_n$ and taking slant product. It suffices to analyze the action of the Galois action of $\sigma$ on the $\ell$-adic cohomology of $B\mathrm{PGL}_n$.

As all the tautological classes on $B\mathrm{PGL}_n$ lie in the image of the cycle class map, we have 
\[
\bigoplus_{i} H^{2i}(B\mathrm{PGL}_n, \BQ_\ell(i))  = \BQ_\ell[1, \mathrm{ch}'_2, \mathrm{ch}'_3,\cdots, \mathrm{ch}'_n], \quad \mathrm{ch}'_i \in H^{2i}(B \mathrm{PGL}_n, \BQ_\ell(i)).
\]
Here $\BQ_\ell(i)$ is the Tate-twist, and the class $\mathrm{ch}'_i$ is the image of the corresponding class in the rational Chow group via the cycle class map. In particular, each class $\mathrm{ch}'_i$ is Galois invariant,
\[
\sigma(\mathrm{ch}'_i) = \mathrm{ch}'_i.
\]
We twist $\mathrm{ch}'_i$ by a basis $t_\ell$ of $\BQ_\ell(1)$:
\[
\mathrm{ch}_i: = t^{-i}_\ell \cdot \mathrm{ch}'_i \in H^{2i}(B \mathrm{PGL}_n, \BQ_\ell).
\]
The Galois element $\sigma$ on the \emph{untwisted} $\ell$-adic cohomology is 
\[
\sigma: H^*(B\mathrm{PGL}_n, \BQ_\ell) \xrightarrow{\simeq}H^*(B\mathrm{PGL}_n, \BQ_\ell), \quad \mathrm{ch}_i \mapsto \chi^{-i}_{\ell}(\sigma) \cdot \mathrm{ch}_i
\]
with $\chi_\ell$ the $\ell$-adic cyclotomic character. 


Finally, we consider the Galois action on the character varieties (\ref{Gal-coh}). The discussion above for $B\mathrm{PGL}_n$ has already shown that the Galois action (\ref{Gal-coh}) satisfies 
\[
\alpha_i \mapsto \chi^{-i} \alpha_i, \quad \beta_i \mapsto \chi^{-i} \beta_i, \quad \psi_{i,j} \mapsto \chi^{-i}\psi_{i,j}
\]
with $\chi \in \BQ^*_\ell$ a non-zero constant. On the other hand, by extending the bigraded decomposition of Theorem \ref{thm1.2} to $\BQ_\ell$-coefficients, there is an automorphism of $\BQ_\ell$-algebras
\begin{equation}\label{aut_thm1.2}
H^*(\widehat{M}_{n,d'}, \BQ_\ell) \xrightarrow{\simeq}H^*(\widehat{M}_{n,d'}, \BQ_\ell)
\end{equation}
which sends $\gamma \in H_{(i)}^*(\widehat{M}_{n,d'}, \BQ_\ell)$ to $\chi^i \gamma$. The desired isomorphism $\varphi_{d,d'}$ can be constructed as the composition of the Galois action (\ref{Gal-coh}) and the automorphism (\ref{aut_thm1.2}) with $\BK = \BQ_\ell$. \qed

From now on, we define the graded $\BQ$-algebra
\[
R_{g,n}: = (H^*(\widehat{M}^B_{n,d}, \BQ), ~~ \cup);
\]
the right-hand side is $d$-independent by the canonical isomorphisms $\varphi_{d,d'}$. This $\BQ$-algebra is bigraded
\[
R_{g,n} = \bigoplus_{k,i} R^k_{g,n,(i)}, \quad R^k_{g,n,(i)} = H^k_{(i)}(\widehat{M}_{n,d}, \BQ),
\]
and is generated by the tautological classes (\ref{PGL_taut}) with Chern grading given by (\ref{Chern_grading}); the ideal of relations is homogeneous with respect to the Chern grading and is $d$-independent.

\subsection{Homotopy groups}\label{Homotopy}

We conclude this section by showing that homotopy groups are not sufficient to distinguish Galois conjugate character varieties either.

\begin{prop}
For any $k>1$ and $d,d'$ coprime to $n$, we have 
\[
\pi_k(M_{n,d}) \simeq \pi_k(M_{n,d'}).
\]
\end{prop}

\begin{proof}
Since the statement is trivial when $n\leq 4$, we assume that $n\geq 5$. Recall the topological description of $M_{n,d}$ from Section \ref{sec1.1}:
\[
M_{n,d} = ( M_{1,0} \times \widecheck{M}_{n,d})\big{/}\Gamma,
\]
where the diagonal action of $\Gamma$ on $ M_{1,0} \times \widecheck{M}_{n,d}$ is free. By the proof of \cite[Theorem 4.1]{BGPG}, the fundamental group of $\widecheck{M}_{n,d}$ coincides with the fundamental group of the moduli space of stable vector bundles with fixed determinant of degree $d$; by \cite[Theorem 3.2(i)]{DU} the latter is trivial. We conclude that $\widecheck{M}_{n,d}$ is simply-connected. Moreover, since $M_{1,0}$ is homeomorphic to $\BR^{2g} \times (S^1)^{2g}$ and $\Gamma$ is a finite group, the universal covering of $M_{n,d}$ is 
\[
\BR^{4g} \times \widecheck{M}_{n,d}. 
\]
Therefore, we have
\[
\pi_k(M_{n,d}) \simeq \pi_k( \widecheck{M}_{n,d}), \quad k \geq 2.
\]
The right-hand side has the homotopy type of a finite CW-complex, whose higher homotopy groups are finitely generated abelian groups. In particular, the group $\pi_k(\widecheck{M}_{n,d})$ is completely determined by its pro-finite completion, which is invariant for Galois conjugate simply-connected varieties $\widecheck{M}^B_{n,d}$ and $\widecheck{M}^B_{n,d'}$; see \cite{Sull}. Hence
\[
\pi_k(M_{n,d}) \simeq \pi_k( \widecheck{M}_{n,d}) \simeq \pi_k( \widecheck{M}_{n,d'}) \simeq \pi_k({M}_{n,d'}), \quad k \geq 2.
\]
The proof is complete.
\end{proof}

\section{Integral cohomology and proof of Theorem \ref{thm0.5}}

The purpose of this section is to prove Theorem \ref{thm0.5}. We assume in this section that $n\geq 5$, as Theorem \ref{thm0.5} (or more precisely, Conjecture \ref{main_conj}) is trivial for $n\leq 4$. We fix a curve $C$ of genus $g\geq 2$. Let $d,d'$ be two integers coprime to $n$ which do not satisfy (\ref{condition}). We prove by contradiction; assume that there is an isomorphism of graded $\BQ$-algebras
\[
\Phi_{d,d'}: H^*(M_{n,d}, \BQ) \xrightarrow{\simeq} H^*(M_{n,d'}, \BQ)
\]
which preserves the integral structures
\[
\Phi_{d,d'}: H^*(M_{n,d}, \BZ) \xrightarrow{\simeq} H^*(M_{n,d'}, \BZ).
\]

\subsection{From $\mathrm{GL}_n$ to $\mathrm{PGL}_n$}\label{sec2.1}

In view of the product (\ref{coh_GL/PGL}), the essential cohomological information of $M_{n,d}$ is encoded in $R_{g,n}$ given by the rational cohomology of the $\mathrm{PGL}_n$ moduli space $\widehat{M}_{n,d}$. On the other hand, it is more convenient to work with the integral structure of the cohomology of $M_{n,d}$ (as the space $\widehat{M}_{n,d}$ is not a manifold but an orbifold).

Our first step is to endow $R_{g,n}$ with a sequence of integral structures given by $M_{n,d}$ with various $d$.


Recall that under the product (\ref{coh_GL/PGL}), the degree 1 cohomology is completely governed by the first factor
\begin{equation}\label{2.1_1}
H^1(M_{n,d}, \BQ) = H^1(M_{1,0}, \BQ) \otimes H^0(\widehat{M}_{n,d}, \BQ). 
\end{equation}
We consider the ideal 
\[
H^*(M_{n,d}, \BQ)^{\langle 1 \rangle} \subset H^*(M_{n,d}, \BQ)
\]
generated by $H^1(M_{n,d}, \BQ)$. By (\ref{2.1_1}) and the fact that $H^*(M_{1,0}, \BQ) = \bigwedge^\bullet H^1(M_{1,0}, \BQ)$, we have a natural isomorphism of graded $\BQ$-algebras
\begin{equation}\label{natural_quotient}
H^*(M_{n,d}, \BQ)/H^*(M_{n,d}, \BQ)^{\langle 1 \rangle}  =  H^*(\widehat{M}_{n,d}, \BQ).
\end{equation}
Using the quotient map
\begin{equation*}
H^*(M_{n,d}, \BQ) \twoheadrightarrow  R_{g,n} = H^*(\widehat{M}_{n,d}, \BQ),
\end{equation*}
the integral structure $H^*(M_{n,d}, \BZ) \subset H^*(M_{n,d}, \BQ)$ induces a lattice
\[
R_{g,n,d,\BZ} \subset R_{g,n}.
\]
We note that, although $R_{g,n}$ is $d$-independent by Theorem \ref{thm0.3}, the lattice $R_{g,n,d,\BZ}$ clearly depends on $d$.

By (\ref{natural_quotient}), the graded isomorphism
\[
\Phi_{d,d'}: H^*(M_{1,0}, \BQ) \otimes R_{g,n} \xrightarrow{\simeq} H^*(M_{1,0}, \BQ) \otimes R_{g,n}
\]
induces a graded automorphism
\[
\widehat{\Phi}_{d,d'} \in G_{g,n} = \mathrm{Aut}^{\mathrm{gr}}(R_{g,n})
\]
preserving the lattices,
\begin{equation}\label{integral}
\widehat{\Phi}_{d,d'}(R_{g,n,d,\BZ}) = R_{g,n,d',\BZ}.
\end{equation}
The purpose of the remainder is to show that, under mild assumptions (\emph{i.e.} weakly diagonal) on $\widehat{\Phi}_{d,d'}$, (\ref{integral}) cannot happen.

\subsection{Integral tautological classes}\label{sec2.2}
The integral cohomology $H^*(M_{n,d}, \BZ)$ is torsion free by \cite[Theorem 6.1]{GS}, and the tautological generators were found by Markman \cite{Markman-integral}. In this section, we need precise information for some low degree integral tautological classes.

We first fix some notation. Recall the cohomology classes for the curve $C$:
\begin{equation}\label{classes_2.2}
\mathbf{1}_C,~~ \mathrm{pt}_C,~~ \delta_i \in H^*(C,\BZ), \quad i=1,2,\cdots, 2g.
\end{equation}
The natural isomorphisms between the topological $K$-theory and the integral cohomology
\[
K_{\mathrm{top}}^0(C) = H^0(C,\BZ) \oplus H^2(C,\BZ), \quad K_{\mathrm{top}}^1(C) = H^1(C, \BZ),
\]
 lift the classes (\ref{classes_2.2}):
\[
[\mathbf{1}_C],~~ [\mathrm{pt}_C] \in K_{\mathrm{top}}^0(C), \quad\quad [\delta_i] \in K_{\mathrm{top}}^1(C), ~~ i=1,2,\cdots, 2g.
\]

We start with recalling the generalized theta divisor in $H^2(M_{n,d}, \BZ)$. Universal bundle $\CU_d$ over $C\times M^{\mathrm{Dol}}_{n,d}$ is not unique; we may normalize it as follows. Consider the integer $\chi: =d+(1-g)n$, which is coprime to $n$. There exist coprime integers $a,b$ with
\begin{equation}\label{a_b}
an+b\chi = 1
\end{equation}
which we fix from now on. We also pick a point $c_0\in C$. We define a new universal bundle $\CU^\circ_d$ from $\CU_d$:
\[
\CU^\circ_d: = \CU_d \otimes p^*_M\CL^{-a}_{c_0} \otimes p_M^*\CL^{ -b}_{\mathrm{det}}.
\]
Here $\CL_{c_0}, \CL_{\mathrm{det}}$ are line bundles on $M_{n,d}$ given by
\[
\CL_{c_0} := \mathrm{det}(\CU_d{|}_{c_0 \times M_{n,d}}), \quad \quad \CL_{\mathrm{det}}:= \mathrm{det}(Rp_{M*} \CU_d).
\]
If we replace $\CU_d \mapsto \CU_d\otimes p_M^*L$ with $L$ any line bundle on $M_{n,d}$, then 
\[
\CL_{c_0} \mapsto \CL_{c_0} \otimes L^n, \quad \CL_{\mathrm{det}} \mapsto \CL_{\mathrm{det}}\otimes L^\chi.
\]
Therefore $\CU_d^\circ \mapsto \CU^\circ_d$; \emph{i.e.}, the bundle $\CU^\circ_d$ is independent of the choice of the universal bundle $\CU_d$ with which we start.

The $K$-theoretic K\"unneth decomposition of the universal bundle $\CU^\circ_d$ reads
\begin{equation}\label{eqn: k-kunneth}
    [\CU^\circ_d]= e_{[\mathbf{1}_C]}\otimes [\mathbf{1}_C]+  e_{[\mathrm{pt}_C]}\otimes [\mathrm{pt}_C] +    \sum_{i=1}^{2g} e_{[\delta_i]} \otimes [\delta_i].
\end{equation}

\begin{thm}[\cite{Markman-integral}] \label{Markman_integral} 
The Chern classes of the K\"unneth factors 
\[
e_{[\#]} \in K_{\mathrm{top}}^*(M_{n,d}), \quad   \#\in \{1_C, \delta_i,\mathrm{pt}_C\}
\]
generate the integral cohomology $H^*(M_{n,d},\mathbb{Z})$ as a graded algebra.
\end{thm}

\begin{rmk}
\begin{enumerate}
    \item In this section we only need Theorem \ref{Markman_integral} for the low degree cohomology $H^{\leq 4}(M_{n,d}, \BZ)$. This may be deduced directly from the corresponding result of Atiyah--Bott for the moduli space of vector bundles \cite{AB} and the codimension estimate \cite{BGPG}.
    \item For a class in topological $K$-theory, there are non-trivial Chern classes indexed by half-integers. These classes are needed to yield generators in odd cohomology in Theorem \ref{Markman_integral}.
\end{enumerate}
    
\end{rmk}

Let $\CF$ be a vector bundle on $C$  with
\[
\mathrm{rk}(\CF) = n, \quad  \mathrm{deg}(\CF) =-\chi.
\]
The associated generalized theta line bundle is defined to be
\begin{equation*}\label{eqn: def Theta}
    \Theta_\CF:=\mathrm{det}Rp_{M*}(\CU^\circ_d\otimes p_C^*\CF)^{-1} \in \mathrm{Pic}(M_{n,d})
\end{equation*}
with $p_{C}, p_M$ the natural projections. By \cite[Proposition 4.5]{Alper}, the class
\[
\Theta:= c_1(\Theta_\CF) \in H^2(M_{n,d}, \BZ)
\]
is independent of $\CF$.

The following lemma expresses $\Theta$ in terms of the tautological classes of Theorem \ref{Markman_integral}.

\begin{lem}\label{lem2.2}
We have
\[
\Theta = d \; c_1(e_{[\mathbf{1}_C]}) -n\; c_1(e_{[\mathrm{pt}_C]}) \in H^2(M_{n,d}, \BZ).
\]
 Equivalently, we have
     \begin{equation}\label{lem2.2_id0}
    \Theta = \chi\; c_1(\CU^\circ_d{|}_{c_0 \times M_{n,d}}) -n \; c_1(Rp_{M*} \CU^\circ_d) \in H^2(M_{n,d}, \BZ).
    \end{equation}
\end{lem}
\begin{proof}
    Applying $p_{M!}(-\otimes p_C^*\CF)$ to \eqref{eqn: k-kunneth}, we obtain that 
    \begin{align*}
        p_{M!}(\CU_d^{\circ}\otimes p_C^*\CF)&=e_{[\mathbf{1}_C]}\chi(\CF) +e_{[\mathrm{pt}_C]}\chi(\mathrm{pt}_C\cdot \CF)\\
        &=-d\; e_{[\mathbf{1}_C]}+n\; e_{[\mathrm{pt}_C]}.
    \end{align*}
    where the first equality follows from the projection formula, and the second equality follows from applying Riemann-Roch to $\CF$ and the fact that $\mathrm{pt}_C\cdot \CF$ is represented by a sheaf of rank $n$ supported at a point on $C$. This proves the first identity. The second identity follows from the first by noticing that
    \begin{equation*}\label{K-theory0}
    e_{[\mathbf{1}_C]} = [\CU^\circ_d{|}_{c_0 \times M_{n,d}}], \quad  e_{[\mathrm{pt}_C]} = [Rp_{M*} \CU_d^\circ] -(1-g)[\CU^\circ_d{|}_{c_0 \times M_{n,d}}]. \qedhere
    \end{equation*}
\end{proof}

Next, we calculate the first Chern class of $\CU^\circ_d$; this is crucial for the normalization calculations associated with $\CU^\circ_d$ later.

\begin{prop}\label{prop2.3}
   The class $c_1(\CU^\circ_d) \in H^2(C\times M_{n,d}, \BZ)$ satisfies
   \begin{equation*}\label{eqn: c1U}
    c_1(\CU^\circ_d)= d \;p_C^*\mathrm{pt}_C+b \; p_{M}^*\Theta+ \text{terms~in~}H^1(C, \BZ)\otimes H^1(M_{n,d}, \BZ).
\end{equation*}
\end{prop}

\begin{proof}
    The first term is clear. For the second term, we note that the definition of $\CU^\circ_d$ gives
    \[
    \mathrm{det}(\CU^\circ_d{|}_{c_0 \times M_{n,d}})^a \otimes \mathrm{det}(Rp_{M*} \CU^\circ_d)^b \simeq \CO_{M_{n,d}},
    \]
    which implies that
    \begin{equation}\label{lem2.1_id}
    a c_1(\CU^\circ_d{|}_{c_0 \times M_{n,d}}) +b c_1(Rp_{M*} \CU^\circ_d) = 0.
    \end{equation}
    Therefore, we have
    \begin{equation}\label{c_1(e)}
     c_1(\CU^\circ_d{|}_{c_0 \times M_{n,d}}) = b\chi c_1(\CU^\circ_d{|}_{c_0 \times M_{n,d}}) -bn c_1(Rp_{M*} \CU^\circ_d) =   b \Theta 
     \end{equation}
     where the first identity is implied by (\ref{lem2.1_id}) and (\ref{a_b}), and the second identity is given by (\ref{lem2.2_id0}).
\end{proof}




For our purpose, we discuss the interaction between the normalized tautological classes (of Section \ref{sec1.2}) and the integral tautological classes in $H^4(M_{n,d}, \BZ)$.


Recall the normalization condition (\ref{eqn: alpha}). By Proposition \ref{prop2.3} we immediately have
\begin{equation}\label{eqn: two factors of alpha}
    \alpha_C=-\frac{d}{n}\mathrm{pt}_C, \quad  \alpha_M=-\frac{b}{n}\Theta.
\end{equation}
The next proposition expresses the integral classes
\[
c_2(e_{[\mathbf{1}_C]}), \quad c_2(e_{[\mathrm{pt}_C]}) \in H^4(M_{n,d}, \BZ)
\]
in terms of the normalized tautological classes
\[
c_2(\mathrm{pt}_C), \quad c_3(\mathbf{1}_C) \in H^4(M_{n,d}, \BQ)
\]
and the generalized theta divisor $\Theta \in H^2(M_{n,d}, \BZ)$.

\begin{prop}\label{prop2.4}
    We have the following identities:
\begin{equation}\label{eqn: kunneth c2 and normalized ch2}
    c_2(e_{[\mathbf{1}_C]})= \frac{n-1}{2n}b^2\Theta^2- c_2(\mathrm{pt}_C),
\end{equation}
\begin{equation}\label{eqn: complicated c2 kunneth}
    c_2(e_{[\mathrm{pt}_C]}) =\kappa_d\Theta^2-\frac{d}{n}  c_2(\mathrm{pt}_C) - c_3(\mathbf{1}_C), 
\end{equation}
where the constant $\kappa_d$ is given by
\begin{equation*}\label{eqn: md}
    \kappa_d=\frac{1}{2n^2}(1+(d-1)b(bd-2)).
\end{equation*}
\end{prop}

\begin{proof}


The proof is essentially applying Riemann--Roch combined with the normalization formulas (\ref{eqn: two factors of alpha}).

We first show \eqref{eqn: kunneth c2 and normalized ch2}. By definition
\[
\mathrm{ch}_2^{\alpha}(\CU^\circ_d)=\mathrm{ch}_2(\CU_d^{\circ})+\alpha c_1(\CU^\circ_d)+\frac{n}{2}\alpha^2, \quad \alpha = p_C^*\alpha_C +p_M^*\alpha_M.
\]
The normalized class $c_2(\mathrm{pt}_C)$ is given by the K\"unneth factor of $\mathrm{ch}^\alpha_2(\CU^{\circ})$ in 
\[
H^0(C,\BQ)\otimes H^4(M_{n,d}, \BQ). 
\]
Therefore, by Proposition \ref{prop2.3} which calculates $c_1(\CU^\circ_d)$, we have
\[
c_2(\mathrm{pt}_C) = \mathrm{ch}_2(e_{[\mathbf{1}_C]}) - \frac{b^2}{n} \Theta^2 + \frac{b^2}{2n} \Theta^2 = -c_2(e_{[\mathbf{1}_C]}) + \frac{n-1}{2n} b^2 \Theta^2
\]
where we have used $c_2 = \frac{1}{2}c_1^2 - \mathrm{ch}_2$ in the second identity. This proves (\ref{eqn: kunneth c2 and normalized ch2}).

We now show \eqref{eqn: complicated c2 kunneth}. By the first equation of Lemma \ref{lem2.2} and (\ref{c_1(e)}), we obtain
\[
c_1(e_{[\mathrm{pt}_C]}) = \frac{bd-1}{n} \Theta.
\]
Again, by definition
\[
\mathrm{ch}_3^{\alpha}(\CU^\circ_d)=\mathrm{ch}_3(\CU^\circ_d)+\alpha\mathrm{ch}_2(\CU^\circ_d)+\frac{\alpha^2}{2}c_1(\CU_d^{\circ})+\frac{n\alpha^3}{6}.
\]
Using \eqref{eqn: two factors of alpha}, we have that 
\begin{equation*}\label{eqn: ch3hat}
      c_3(\mathbf{1}_C)  = \mathrm{ch}_2(e_{[\mathrm{pt}_C]})+\Big(-\frac{b}{n}\Theta c_1(e_{[\mathrm{pt}_C]})-\frac{d}{n}\mathrm{ch}_2(e_{[\mathbf{1}_C]})\Big)+\frac{3 b^2d}{2n^2}\Theta^2-\frac{b^2d}{2n^2}\Theta^2
\end{equation*}
from which \eqref{eqn: complicated c2 kunneth} follows. 
\end{proof}

\begin{rmk}
    The rational cohomology $H^*(M_{n,d}, \BQ)$ is bigraded by Theorem \ref{thm0.3}. However, the Chern grading does not exist for the integral cohomology $H^*(M_{n,d}, \BZ)$; it is essential to use the normalized tautological classes, which are the natural ones induced by the classes pulled back from $H^*(B\mathrm{PGL}_n, \BQ)$, to describe the Chern grading, and those classes are not integral.
\end{rmk}

\subsection{Proof of Theorem \ref{thm0.5}}
The discussion in Section \ref{sec2.2} mainly concerns the cohomology of $M_{n,d}$. In order to complete the proof of Theorem \ref{thm0.5}, we pass these structures to the smaller graded $\BQ$-algebra 
\[
R_{g,n} = H^*(\widehat{M}_{n,d}, \BQ),
\]
and the lattice $R_{g,n,d,\BZ} \subset R_{g,n}$ introduced in Section \ref{sec2.1}. The cohomological grading of $R_{g,n}$ naturally endows a grading on $R_{g,n,d,\BZ}$. By the natural quotient map (\ref{natural_quotient}), classes in $H^*(M_{n,d}, \BZ)$ yield classes in $R_{g,n,d,\BZ}$; we use $\theta_d$ to denote the class in $R^2_{g,n,d,\BZ}$ induced by $\Theta$, and $\CA_d, \CB_d$ to denote the classes in $R^4_{g,n,d,\BZ}$ induced by $c_2(e_{[\mathbf{1}_C]}), c_2(e_{[\mathrm{pt}_C]})$ respectively.

\begin{lem}\label{lem2.6}
    We have
    \[
    R^2_{g,n,d,\BZ} = \BZ \theta_d, \quad R^4_{g,n,d,\BZ} = \BZ \theta^2_d \oplus \BZ \CA_d  \oplus \BZ \CB_d.
    \]
\end{lem}

\begin{proof}
    We first show that $\theta_d$ generates $R^2_{g,n,d,\BZ}$. By Theorem \ref{Markman_integral}, the second integral cohomology $H^2(M_{n,d}, \BZ)$ is spanned by $c_1(e_{[\mathbf{1}_C]}), c_1(e_{[\mathrm{pt}_C]})$ and integral classes lying in the ideal generated by $H^1(M_{n,d}, \BZ)$. Therefore, after taking the quotient (\ref{natural_quotient}), we know that $R^2_{g,n,d,\BZ}$ is spanned by images of $c_1(e_{[\mathbf{1}_C]}), c_1(e_{[\mathrm{pt}_C]})$. Furthermore, by Lemma \ref{lem2.2} and (\ref{c_1(e)}), both classes are proportional to $\Theta$:
    \[
    c_1(e_{[\mathbf{1}_C]}) = b\Theta, \quad c_1(e_{[\mathrm{pt}_C]}) = \frac{bd-1}{n} \Theta.
    \]
    We conclude that $\theta_d$ is sufficient to span $R^2_{g,n,d,\BZ}$.

    For degree $4$, Theorem \ref{Markman_integral} again implies that $H^4(M_{n,d}, \BZ)$ is spanned by \[
    c_2(e_{[\mathbf{1}_C]}),\quad  c_2(e_{[\mathrm{pt}_C]}), \quad \Theta^2
    \] 
 and integral classes lying in the ideal generated by $H^1(M_{n,d}, \BZ)$. Therefore $R^4_{g,n,d,\BZ}$ is spanned by $\CA_d,\CB_d, \theta^2_d$. Finally, there are no relations among these three classes by \cite[Lemma 4.1.12]{HRV}.
\end{proof}

We are now ready to prove Theorem \ref{thm0.5}.

By definition, the images of the (rational!) normalized classes 
\[
c_3(\mathbf{1}_C),~~ c_2(\mathrm{pt}_C) \in H^4(M_{n,d}, \BQ)
\]
under the quotient map (\ref{natural_quotient}) recover
\[
\alpha_3,~~  \beta_2 \in R^4_{g,n}.
\]
Again, \cite[Lemma 4.1.12]{HRV} implies that there are no relations among these classes, so that
\begin{equation}\label{basis_R4}
R^4_{g,n} =  \BQ \alpha^2_2 \oplus \BQ\alpha_3 \oplus \BQ \beta_2;
\end{equation}
we note that $\alpha_2$ is clearly proportional to $\theta_d$ as both are generators of the 1-dimensional $\BQ$-vector space $R^2_{g,n}$.

Now, we prove Theorem \ref{thm0.5} by contradiction: assume that there is an automorphism $\widehat{\Phi}_{d,d'} \in G_{g,n}$ as in Section \ref{sec2.1} satisfying (\ref{integral}). 

First, this automorphism has to preserve $\theta_d$ up to a sign,
\[
\widehat{\Phi}_{d,d'}(\theta_d) = \pm \theta_{d'};
\]
in particular, we have
\[
\widehat{\Phi}_{d,d'}(\theta_d^2) = \theta^2_{d'}.
\]

Furthermore, since by assumption $G^{\mathrm{wd}}_{g,n} = G_{g,n}$, we know that 
\[
\widehat{\Phi}_{d,d'}(\alpha_3) = \mu \alpha_3,\quad \widehat{\Phi}_{d,d'}(\beta_2) = \nu \beta_2, \quad \mu, \nu \in \BQ^*.
\]

We need the following lemma to rule out the existence of $\widehat{\Phi}_{d,d'}$.

\begin{lem}\label{lem: faster lattice lemma}
Let $d,d',b,b'$ be integers.
    We consider the lattices
    \[L:=\BZ\left\langle (u,0,0), (\frac{n-1}{2n}ub^2,0, -1),(u\kappa_d,-1,-\frac{d}{n}) \right\rangle\subset \BQ^3\]
    and
     \[L':=\BZ\left\langle (u',0,0), (\frac{n-1}{2n}u'{b'}^2,0,-1),(u'\kappa_{d'},-1, -\frac{d'}{n}) \right\rangle\subset \BQ^3\]
    where $\kappa_d, \kappa_{d'}$ are defined as in Proposition \ref{prop2.4}, and $u,u'$ are non-zero rational numbers. Assume that the group $(\BQ^*)^3$ acts on $\BQ^3$ diagonally with respect to the standard coordinate basis. If there exists an element in $(\BQ^*)^3$ that sends $L$ to $L'$, then (\ref{condition}) holds.
\end{lem}
\begin{proof}
    After passing to the projection to the last two factors, it suffices to prove the following statement: if there exists $(\lambda,\mu)\in (\BQ^*)^2$ that sends $\BZ\langle(0,1),(1,\frac{d}{n})\rangle$ to $\BZ\langle(0,1),(1,\frac{d'}{n})\rangle$, then (\ref{condition}) holds. Indeed, since $\lambda$-scaling has to preserve $\BZ\subset \BQ$ to itself. We must have $\lambda=\pm 1$. Intersecting with the line 
    \[
     \{0\} \times \BQ \subset \BQ^2, 
    \]
    we see that $\mu$-scaling also has to preserve $\BZ \subset \BQ$; thus $\mu=\pm 1$. There exist integers $t_1,t_2\in \BZ$ so that
    \[
    \left( \lambda,  \mu\frac{d}{n} \right)
    =\left(t_2, t_1+t_2\frac{d'}{n}\right),\quad \lambda,\mu =\pm 1.
    \]
    We conclude that 
    \[
\mu \frac{d}{n} = \lambda\frac{d'}{n} + t_1, \quad \lambda,\mu =\pm 1
    \]
    which proves the lemma.
\end{proof}

\begin{proof}[Finish proving Theorem \ref{thm0.5}]
    Under the basis (\ref{basis_R4}), the element $\widehat{\Phi}_{d,d'} \in G_{g,n}$ acts diagonally, and it sends the lattice 
\[
\BZ \theta^2_d \oplus \BZ\CA_d \oplus \BZ\CB_d \subset  \BQ \alpha^2_2 \oplus \BQ\alpha_3 \oplus \BQ \beta_2
\]
to the lattice
\[
 \BZ \theta^2_{d'} \oplus \BZ\CA_{d'} \oplus \BZ\CB_{d'}  \subset  \BQ \alpha^2_2 \oplus \BQ\alpha_3 \oplus \BQ \beta_2
\]
by Lemma \ref{lem2.6}.

Proposition \ref{prop2.4} implies that the expressions of the integral generators $\theta^2_d, \CA_d, \CB_d$ in terms of the rational generators $\alpha^2_2, \alpha_3, \beta_2$ are exactly as in the expression of the lattice $L$ in the coordinate basis of $\BQ^3$. Therefore we conclude that $\widehat{\Phi}_{d,d'}$ cannot exist unless (\ref{condition}) holds. The proof of Theorem \ref{thm0.5} is now complete. 
\end{proof}

\section{A distinguished tautological relation and proof of Theorem \ref{thm0.8}}\label{sec3}

The goal of this subsection is to prove Theorem \ref{thm0.8}. As a corollary, we obtain in Theorem \ref{thm: the unique relation} a closed formula for the relation $\Upsilon_{2,5}=0$. Since only this relation will be used in Section \ref{sec4}, the reader may skip Section \ref{sec3} and proceed directly to Section \ref{sec4}, assuming Theorem \ref{thm: the unique relation}.

Throughout this section, we assume that $g\geq 2$ and $n\geq 3$. The special rank $n=2$ case will be discussed briefly in Remark \ref{rank2}.

\subsection{Poincar\'e polynomials}

As we discussed in Section \ref{sec1.2}, the cohomology $H^*(\widehat{M}_{n,d}, \BQ)$ is generated by the tautological classes (\ref{PGL_taut}). If there were no relations among these classes, the Poincar\'e polynomial for $\widehat{M}_{n,d}$ would be given by the product formula
\[
\widehat{P}^{\mathrm{free}}(t) = \prod_{i= 2}^{n} \frac{(1+t^{2i-1})^{2g}}{(1-t^{2i-2})(1-t^{2i})}. 
\]
Here the numerator is contributed by the exterior algebra generated by the odd classes 
\[
\psi_{i,j} \in H^{2i-1}(\widehat{M}_{n,d}, \BQ), \quad 1\leq j\leq {2g}, 
\]
and the denominator is contributed by the powers of the even classes
\[
\alpha_i \in H^{2i-2}(\widehat{M}_{n,d}, \BQ), \quad \beta_i \in H^{2i}(\widehat{M}_{n,d}, \BQ).
\]
By \eqref{coh_GL/PGL}, the Poincar\'e polynomial for $M_{n,d}$ would be given by 
\begin{equation}\label{P_free}
P^{\mathrm{free}}(t) = (1+t)^{2g}\widehat{P}^{\mathrm{free}}(t) = (1+t)^{2g}\prod_{i= 2}^{n} \frac{(1+t^{2i-1})^{2g}}{(1-t^{2i-2})(1-t^{2i})}.
\end{equation}
Obviously, this is not possible, since the formula (\ref{P_free}) has infinitely many terms; but comparing this formula with the genuine Poincar\'e polynomial
\[
P(t): = \sum_{i} \dim H^i(M_{n,d}, \BQ)\; t^i \in \BZ[t] 
\]
detects the lowest degree where a relation occurs.

The first result of this section is the following. Recall the constant $r_{g,n} := 4g(n-1)-2n$.

\begin{thm}\label{thm3.1}
    Assume that the curve $C$ has genus $g\geq 2$, and assume that the rank satisfies $n\geq 3$. We have
    \[
    P(t) = P^{\mathrm{free}}(t) - t^{r_{g,n}} + O(t^{r_{g,n}+1}).
    \]
\end{thm}

We prove Theorem \ref{thm3.1} in Sections \ref{sec3.2} and \ref{sec3.3}; it is clear that this theorem implies Theorem \ref{thm0.8}(a,b).

\subsection{The formulas of Hausel--Rodriguez--Villegas and Mellit}\label{sec3.2}

Our main tool for proving Theorem \ref{thm3.1} is the closed formula for $P(t)$ conjectured by Hausel--Rodriguez-Villegas \cite{HRV} (which we call the HRV formula). Mellit \cite{Mellit} proved the HRV formula by reducing it to Schiffmann's work \cite{Sch} where the latter gives a more complicated formula for $P(t)$.

We first recall the explicit form of the HRV formula, which will be crucial for the proof of Theorem \ref{thm3.1}. We begin by introducing some notation for partitions.

For a partition $\mu$ of an integer and any cell $\Box$, we denote by 
\[
a(\Box), ~~ \ell(\Box) \in \BZ_{\geq 0}
\]
the lengths of the arm and the leg of $\Box$ with respect to $\mu$. We use $|\mu|$ to denote the size of the partition $\mu$
; we write $\mu \vdash n$ if $n = |\mu|$.

For $a,\ell \in \BZ_{\geq 0}$, we define
\[
\BB_{a,\ell}(t,z):= \frac{(z^{a+1} +t^{-2\ell-1})^g(z^a+ t^{-2\ell-1})^g}{(z^{a+1}-t^{-2\ell})(z^a-t^{-2\ell-2})}.
\]
We further define the generating series
\[
\Omega(t,z,T) :=  1+ \Omega_1(t,z) T + \Omega_2(t,z) T^2 + \cdots
\]
where each term is given by the following formulas
\[
\Omega_n(t,z): = \sum_{\mu\vdash n} \Omega_{\mu}(t,z), \quad  \Omega_\mu(t,z) := \prod_{\Box  \, \in\, \mu} \BB_{a(\Box), \ell(\Box)}(t,z).
\]
Equivalently, we have
\[
\Omega_g(t,z,T) := \sum_{\mu} T^{|\mu|} \Omega_\mu(t,z).
\]

The following is the HRV formula proven by Mellit \cite{Mellit}.

\begin{thm}[\cite{Mellit}]\label{thm3.2}
    We have
    \[
    P(t) = t^{(2g-2)n^2+2} \, \mathrm{lim}_{z\to 1}\Big{(} -(1-t^{-2})(1-z) [\mathrm{Log}(\Omega(t,z,T))]^n\Big{)}.
    \]
    Here $\mathrm{Log}(-)$ is the plethystic logarithm, $[-]^i$ stands for taking the $T^i$-coefficient, and the limit in the right-hand side is well-defined.
\end{thm}

The plethystic logarithm of $\Omega(t,z,T)$ can be expressed explicitly in terms of $\Omega_i(t,z)$. We record the following lemma which is useful for computing $\mathrm{Log}(\Omega(t,z,T))$. Its proof follows from the definition of the plethystic logarithm.

\begin{lem}
    We have
    \begin{equation}\label{PLog}
    [\mathrm{Log}(\Omega(t,z,T))]^n = [\mathrm{log}(\Omega(t,z,T))]^n +\sum_{d|n,\; d>1} \frac{\mu(d)}{d} \psi_d\left( [\mathrm{log}(\Omega(t,z,T))]^{\frac{n}{d}} \right).
    \end{equation}
    Here $\mathrm{log}(-)$ is the ordinary logarithm, $\mu(-)$ is the M\"obius function, and $\psi_d(-)$ is the Adams operation replacing $t,z$ by $t^d, z^d$ respectively.
\end{lem}

Expanding the ordinary logarithm
\[
\mathrm{log}(\Omega(t,z,T)) = (\Omega(t,z,T) -1) -\frac{1}{2}(\Omega(t,z,T)-1)^2 +\frac{1}{3} (\Omega(t,z,T)-1)^3 - \cdots,
\]
we obtain the \emph{first term} of (\ref{PLog}) to be
\begin{equation}\label{exp_log}
[\mathrm{log}(\Omega(t,z,T))]^n = \sum_{k\geq 1} \frac{(-1)^{k+1}}{k} \sum_{n_1 +n_2+ \cdots + n_k=n} \Omega_{n_1}(t,z)\Omega_{n_2}(t,z) \cdots \Omega_{n_k}(t,z). 
\end{equation}
Each term in the right-hand side can further be expressed in terms of 
\begin{equation}\label{Omega_terms}
\Omega_{\mu_1}(t,z)\Omega_{\mu_2}(t,z)\cdots \Omega_{\mu_k}(t,z), \quad \mu_i \vdash n_i, \quad  \sum_{i=1}^k n_i = n.
\end{equation}
Using this expression, there is a \emph{principal term} in the right-hand side of Theorem \ref{thm3.2} contributed by the single column partition $(1^n) \vdash n$ through (\ref{Omega_terms}), (\ref{exp_log}), and the first term in the right-hand side of (\ref{PLog}):
\[
P^{\mathrm{prin}}(t) :=t^{(2g-2)n^2+2} \, \mathrm{lim}_{z\to 1}\Big{(} -(1-t^{-2})(1-z) \Omega_{(1^n)}(t,z)\Big{)}.
\]

We observe that the principal term is matched exactly with $P^\mathrm{free}(t)$.

\begin{prop}
We have
\[
P^{\mathrm{prin}}(t)  = P^{\mathrm{free}}(t).
\]
\end{prop}

\begin{proof}
    For the column partition $(1^n)$, the $n$ boxes associated with it satisfy
    \[
    (a(\Box), \ell(\Box)) = (0,0),\; (0,1), \; \dots, \: (0,n-1).
    \]
    By definition, we have
    \[
    \Omega_{(1^n)}(t,z) = \BB_{0,0}(t,z)\BB_{0,1}(t,z)\cdots \BB_{0,n-2}(t,z) \BB_{0,n-1}(t,z).
    \]
    Only the term $\BB_{0,0}(t,z)$ has a simple pole at $z=1$. Hence
    \begin{align*}
P^{\mathrm{prin}}(t)
&=
t^{(2g-2)n^2+2}(1-t^{-2})
\left[(z-1)\BB_{0,0}(t,z)\right]_{z=1}
\prod_{\ell=1}^{n-1}\BB_{0,\ell}(t,1)\\
&=
t^{(2g-2)n^2+2}(1+t^{-1})^{2g}
\prod_{\ell=1}^{n-1}
\frac{(1+t^{-2\ell-1})^{2g}}
     {(1-t^{-2\ell})(1-t^{-2\ell-2})}\\
&=
(1+t)^{2g}
\prod_{i=2}^{n}
\frac{(1+t^{2i-1})^{2g}}
     {(1-t^{2i-2})(1-t^{2i})}\\
&=P^{\mathrm{free}}(t). \qedhere
\end{align*}
\end{proof}

If we write the HRV formula as the principal and the correction terms:
\[
P(t) = P^{\mathrm{prin}}(t) + P^{\mathrm{corr}}(t),
\]
it suffices to show that the correction term $P^{\mathrm{corr}}(t)$ satisfies
    \begin{equation}\label{goal_sec3.3}
    P^{\mathrm{corr}}(t) =- t^{r_{g,n}} + O(t^{r_{g,n}+1}).
    \end{equation}
We complete this in the next section.

\subsection{Proof of Theorem \ref{thm3.1}}\label{sec3.3}

In view of the right-hand side of the formula (\ref{PLog}), the correction term $P^{\mathrm{corr}}(t)$ is contributed by 2 parts,
\[
P^{\mathrm{corr}}(t) = P_1^{\mathrm{corr}}(t) + P_2^{\mathrm{corr}}(t),
\]
where the first is the contribution of the ordinary logarithm and the second is the contribution of the summation over $d|n$ with $d>1$.

We first assume $(g,n)\neq (2,4)$, and prove the estimates
\begin{align}
P^{\mathrm{corr}}_1(t) 
    &= -t^{r_{g,n}} + O\bigl(t^{r_{g,n}+1}\bigr),
    \label{corr1}\\
P^{\mathrm{corr}}_2(t)
    &= O\bigl(t^{r_{g,n}+1}\bigr).
    \label{corr2}
\end{align}
These estimates immediately imply \eqref{goal_sec3.3}. We handle the exceptional case $(g,n)=(2,4)$ at the end; see the part C.
\medskip

{\noindent \bf A. Proof of (\ref{corr1}).} 

\medskip

\noindent For an element $f(t,z) \in  \BQ(\! (z-1) \!)(\!(t)\!)$, we define
\[
\mathrm{deg}^t(f(t,z)) \in \BZ
\]
to be its lowest nontrivial $t$-degree. 

Now we consider $\BB_{a,\ell}(t,z)$ which can be viewed as an element in $\BQ(\! (z-1) \!)(\!(t)\!)$.  By definition we have
\[
\BB_{a,\ell}(t,z) = t^{-(2g-2)(2\ell+1)} \overline{\BB}_{a,\ell},
\]
where $\overline{\BB}_{a,\ell}$ is regular on the variable $t$ and has at worst a simple pole at $z=1$. For example
\[
\BB_{1,0}(t,z) = t^{-(2g-2)} \overline{\BB}_{1,0}, \quad      \overline{\BB}_{1,0} =  \frac{1}{(1-z)(1+z)}\cdot \frac{(1+z^2t)^g(1+zt)^g}{1-zt^2}.
\]
In particular, we have
\[
\mathrm{deg}^t ( \BB_{a,\ell}(t,z)) = -(2g-2)(2\ell +1).
\]
Therefore, for a given partition $\mu \vdash n$, we have
\begin{equation}\label{t-deg}
\mathrm{deg}^t(\Omega_{\mu}(t,z)) = -(2g-2)\iota(\mu), \quad \iota(\mu): =\sum_{\Box \; \in \; \mu}(2\ell(\Box)+1).
\end{equation}


\begin{lem}\label{lem3.5}
    Assume that $\mu$ runs through all partitions of $n$.
    \begin{enumerate}
        \item[(a)]  The largest possible value of $\iota(\mu)$ is $n^2$, which is attained uniquely by the single column partition $(1^n)$. 
        \item[(b)] The second largest value of $\iota(\mu)$ is $(n-1)^2+1$, which is attained uniquely by the hook partition $(1^{n-2}2)$.
    \end{enumerate} 
\end{lem}

\begin{proof}
    The easiest way to see both properties is to use the conjugate partition $\mu'$ of $\mu$. Since a column of the partition $\mu$ of height $k$ contributes 
    \[
    1+3+ \cdots + (2k-1) = k^2
    \]
to $\iota(\mu)$, we see that $\iota(\mu)$ is exactly calculating the square of the norm of the conjugate partition
\[
\left\| \mu' \right\|^2 = \iota(\mu). 
\]
When we fix the size, the largest and the second largest norms are given by the partitions $(n)$ and $(1\cdot (n-1))$ respectively; this proves the lemma.
\end{proof}

We consider the expressions (\ref{exp_log}) and (\ref{Omega_terms}). The following estimate is useful for our purpose.

\begin{lem}\label{lem3.51}
    Consider the product (\ref{Omega_terms}). If it is not one of the following 
   \[
   \Omega_{(1^n)}(t,z), \quad \Omega_{(1^{n-2}2)}(t,z), \quad \Omega_{(1^{n-1})}(t,z) \Omega_{(1)}(t,z),
   \]
   then we have
\[
 \iota(\mu_1) +\iota(\mu_2) + \cdots +\iota(\mu_k) \leq 
\begin{cases}
n^2-4n+8, & n\geq 4; \\
3, &  n=3. 
\end{cases}
\]
\end{lem}

\begin{proof}
We first assume $n\geq 4$. 

If $k=1$, then Lemma \ref{lem3.5} gives the largest and the second largest values of $\iota(\mu_1)$, which are both achieved in the cases we exclude. By the proof of Lemma \ref{lem3.5}, the next largest one is
    \[
    \iota(\mu) = 2^2 +(n-2)^2 = n^2-4n+ 8
    \]
    achieved by $(1^{n-4}2^2)$ whose conjugate partition is $(2\cdot (n-2))$.

    If $k>1$, then by Lemma \ref{lem3.5}(a) we have
    \[
    \sum_{i=1}^k \iota(\mu_i) \leq  \sum_{i=1}^k |\mu_i|^2.
    \]
    After excluding the case $k=2$ and $(|\mu_1|, |\mu_2|) = (1,n-1)$, the maximum of the right-hand side under the constraint $\sum_i |\mu_i| = n$ is 
    \[
2^2+    (n-2)^2  = n^2 -4n +8.
    \]
    The proof is complete for $n\geq 4$.

    When $n=3$, the answer is obvious by listing all the possible cases.
\end{proof}

By Lemmas \ref{lem3.5} and \ref{lem3.51}, the lowest $t$-degree among all the terms (\ref{Omega_terms}) is attained by
\[
\mathrm{deg}^t(\Omega_{(1^n)}(t,z)) = -(2g-2)n^2;
\]
this yields exactly $P^{\mathrm{prin}}(t)$. The second lowest $t$-degree is
\begin{equation}\label{t-degree}
-(2g-2)((n-1)^2 +1)
\end{equation}
which can be attained by two terms simultaneously:
\begin{equation}\label{contribution_Omega}
\Omega_{(1^{n-2}2)}(t,z), \quad \Omega_{(1^{n-1})}(t,z) \Omega_{(1)}(t,z).
\end{equation}
The contribution of the two terms (\ref{contribution_Omega}) to $\mathrm{log}(\Omega(t,z,T))$ is
\begin{equation}\label{Omega_B}
\Omega_{(1^{n-2}2)}(t,z) - \Omega_{(1^{n-1})}(t,z) \Omega_{(1)}(t,z) = \BB^2_{0,0}(t,z) \prod_{i=1}^{n-3} \BB_{0,i}(t,z)\left( \BB_{1,n-2}(t,z) -  \BB_{0,n-2}(t,z) \right).
\end{equation}

Interestingly, due to some mysterious cancelation, the lowest $t$-degree of (\ref{Omega_B}) is higher than expected, as we calculate in the following lemma.

\begin{lem}\label{lem3.6}
    We have
    \[
     \BB_{1,n-2}(t,z) -  \BB_{0,n-2}(t,z)  = (z-1)\BB_{0,n-2}(t,1)(t^{2n-4} + O(t^{2n-3})) + O((z-1)^2).
    \]
\end{lem}

\begin{proof}
  It is clear by definition that
\[
 \BB_{1,n-2}(t,1) = \BB_{0,n-2}(t,1).
\]
The first order calculation with respect to the variable $z-1$ gives
\begin{align*}
     \BB_{1,n-2}(t,z) - \BB_{0,n-2}(t,z) &= (z-1)\BB_{0,n-2}(t,1)\Big{(} \partial_z\; \mathrm{log} \BB_{1,n-2}(t,z)|_{z=1} - \partial_z\; \mathrm{log} \BB_{0,n-2}(t,z)|_{z=1}     \Big{)} \\ &  + O((z-1)^2) \\
     & = (z-1)\BB_{0,n-2}(t,1)\left( \frac{t^{2n-4}}{1-t^{2n-4}} + \frac{2gt^{2n-3}}{1+t^{2n-3}} + \frac{t^{2n-2}}{1-t^{2n-2}} \right) \\ & 
     + O((z-1)^2). 
\end{align*}
The lemma follows by the fact that the sum of the three terms in the bracket has the leading term $t^{2n-4}$.
\end{proof}

Combining (\ref{Omega_B}), Lemma \ref{lem3.6}, and Theorem \ref{thm3.2}, the $t$-degree in $P^{\mathrm{corr}}_1(t)$ contributed by (\ref{contribution_Omega}) is 
\begin{equation}\label{total_t_degree}
\Big{(}(2g-2)n^2+2 \Big{)} - 2 - \Big{(}(2g-2)((n-1)^2+1)\Big{)} + \Big{(}2n-4\Big{)} = r_{g,n}
\end{equation}
where the first term is given by the pre-factor $t^{(2g-2)n^2+2}$ of Theorem \ref{thm3.2}, the second term $-2$ is given by the pre-factor $(1-t^{-2})$, the third term is given by (\ref{t-degree}), and the fourth term is given by Lemma \ref{lem3.6}. By a direct calculation, the coefficient of $t^{r_{g,n}}$ contributed by (\ref{contribution_Omega}) is $-1$ where the sign essentially comes from the sign in the formula of Theorem \ref{thm3.2}.


To complete the proof of (\ref{corr1}), we still need to check that any other terms of the form (\ref{Omega_terms}) cannot contribute to $P^{\mathrm{corr}}_1(t)$ in degrees $\leq r_{g,n}$.

Now, under the assumption $(g,n)\neq (2,4)$, we have
\begin{equation}\label{inequal_A}
\Big{(}(2g-2)n^2 +2 \Big{)} -2 - \Big{(} (2g-2) \sum_i \iota(\mu_i) \Big{)} > r_{g,n}.
\end{equation}
This is an elementary exercise using Lemma \ref{lem3.51}. We conclude that the lowest degree term in $P_1^{\mathrm{corr}}(t)$ is given by (\ref{contribution_Omega}); the proof of (\ref{corr1}) is complete. \qed

\begin{rmk}
    When $(g,n)=(2,4)$, (\ref{inequal_A}) is an equality. We will handle this exceptional case in the part C.
\end{rmk}

\medskip

{\noindent \bf B. Proof of (\ref{corr2}).}
\medskip

\noindent The proof of (\ref{corr2}) is completely parallel. We need to analyze the contribution of 
\[
\psi_d\left( [\mathrm{log}(\Omega(t,z,T))]^{\frac{n}{d}} \right), \quad d|n,~~ d>1;
\]
it is governed by the $t$-expansion of 
\[
\psi_d\left( \Omega_{\mu_1}(t,z)\Omega_{\mu_2}(t,z) \cdots \Omega_{\mu_k}(t,z)\right), \quad |\mu_1|+|\mu_2|+ \cdots + |\mu_k|=\frac{n}{d}.
\]
By the definition of the Adams operation, we have
\begin{equation}\label{adam_Omega}
\psi_d\left( \Omega_{\mu_1}(t,z)\Omega_{\mu_2}(t,z) \cdots \Omega_{\mu_k}(t,z)\right) =  \Omega_{\mu_1}(t^d,z^d)\Omega_{\mu_2}(t^d,z^d) \cdots \Omega_{\mu_k}(t^d,z^d).
\end{equation}
By (\ref{t-deg}), the $t$-degree of (\ref{adam_Omega}) is
\[
-(2g-2)d \left(\iota(\mu_1) +\iota(\mu_2)+ \cdots + \iota(\mu_k) \right).
\]
The same discussion as in part A implies that its minimum for fixed $d$ is attained when $k=1$ and $\mu_1 = (1^{\frac{n}{d}})$; in this case 
\begin{equation}\label{t_deg_Adam}
\mathrm{deg}^t\left(\Omega_{(1^{\frac{n}{d}})}(t^d, z^d)\right) = -(2g-2)d \left(\frac{n}{d}\right)^2.
\end{equation}

Hence, as we calculated in (\ref{total_t_degree}), the lowest possible $t$-degree of $P_2^{\mathrm{corr}}(t)$ is
\begin{align*}
    \Xi_{g,n,d}&:= \Big{(} (2g-2)n^2+2 \Big{)} - 2 -\Big{(} (2g-2)d \left(\frac{n}{d}\right)^2  \Big{)}
\end{align*}
where the first two terms are given by the pre-factors in the formula of Theorem \ref{thm3.2} and the third term is given by (\ref{t_deg_Adam}). A direct calculation gives 
\[
\Xi_{g,n,d} = (2g-2)n^2\left(1-\frac{1}{d} \right).
\]
If we allow $d$ to vary, the smallest value of $\Xi_{g,n,d}$ is clearly attained when $d$ is the smallest prime factor of $n$.

\begin{lem}\label{lem3.7}
    For $g\geq 2, n\geq 3$, we have
    \[
    \Xi_{g,n,d} \geq r_{g,n}.
    \]
    The equality is attained if and only if $(g,n,d)= (2,4,2)$.
\end{lem}

\begin{proof}
    If $n$ is odd, then $d\geq 3$, and 
    \[
    \Xi_{g,n,d} \geq (2g-2)n^2\left(1-\frac{1}{3}\right) >  4g(n-1)-2n = r_{g,n}
    \]
    where the second inequality is elementary. 

    When $n$ is even, we have $\Xi_{g,n,d} \geq \Xi_{g,n,2}$, and
    \[
    \Xi_{g,n,2} - r_{g,n} = (n-2)(g(n-2)-n).
    \]
   Since $n\geq 4$ and $g\geq 2$, the right-hand side is always $\geq 0$, and it is $0$ only if $g=2, n=4$.
\end{proof}

Under the assumption $(g,n) \neq (2,4)$, Lemma \ref{lem3.7} shows that $P_2^{\mathrm{corr}}(t)$ does not have any nontrivial term in degree $\leq r_{g,n}$. This completes the proof of (\ref{corr2}). \qed

\medskip

{\bf \noindent C. The exceptional case $(g,n)=(2,4)$.}

\medskip

\noindent We prove Theorem \ref{thm3.1} in the last remaining case. Beyond the contribution of (\ref{contribution_Omega}) whose lowest degree term gives exactly $-t^{r_{2,4}}$ in $P^{\mathrm{corr}}(t)$ as we discussed in the parts A and B, there are three extra terms for $(g,n)=(2,4)$:
\[
\Omega_{(2^2)}(t,z), \quad -\frac{1}{2}(\Omega_{(1^2)}(t,z))^2, \quad -\frac{1}{2}\psi_2(\Omega_{(1^2)}(t,z)).
\]
All these terms contribute to $t^{r_{2,4}}$ in $P^{\mathrm{corr}}(t)$. 

However, although each term contributes nontrivially to the coefficient $t^{r_{2,4}}$, their sum is $0$. More precisely, we list the leading rational factor of each term that contributes to the coefficient of $t^{r_{2,4}}$:
\[
\frac{1}{(z^2-1)(z-1)}, \quad -\frac{1}{2(z-1)^2}, \quad \frac{1}{2(z^2-1)}.
\]
It is straightforward to check that their sum vanishes. 

The proof of Theorem \ref{thm3.1} is complete. \qed

\begin{rmk}\label{rank2}
In both parts A and B, we crucially used the assumption that $n > 2$. When $n=2$, Theorem \ref{thm0.8}(a,b) does not hold: the lowest degree in which nontrivial relations appear is $4g-2 = r_{g,2}+2$, and there are $\binom{2g}{2}$ independent relations. This can be verified either using the HRV formula for the Poincar\'e polynomial \cite{HRV}, or using the explicit presentation of the tautological relations due to Hausel--Thaddeus \cite{HT2}. Since we focus on the higher-rank cases in this paper, we leave this to the interested reader.
\end{rmk}

\subsection{Proof of Theorem \ref{thm0.8}}

We have already proven Theorem \ref{thm3.1}, which implies Theorem \ref{thm0.8}(a,b). To prove the part (c), we need a method to find the relation in degree $r_{g,n}$. Our key tool is the \emph{Chern grading} of Theorem \ref{thm0.3} and the relations of Earl--Jeffrey--Kirwan \cite{EK, JK} for the moduli space of stable vector bundles.

Our strategy is the following: the two ingredients above can help us to find the set of candidates of the relations (\emph{i.e.} the homogeneous EJK relations which we will define later) for any $g,n$; if for certain $g,n$ we are lucky enough so that there is only \emph{one} homogeneous EJK relation in the degree $r_{g,n}$, Theorem \ref{thm0.8}(b) can guarantee that it has to be a genuine relation.

As in Section \ref{sec1.1}, we fix a line bundle $L$ of degree $d$ on the curve. Let $N_{n,L}$ be the moduli space of stable vector bundles $\CE$ on $C$ with
\[
\mathrm{rk}(\CE) =n ,\quad \mathrm{det}(\CE) \simeq L.
\]
Recall the Dolbeault moduli space $\widecheck{M}_{n,d}^{\mathrm{Dol}}$ associated with the line bundle $L$; we have a natural embedding 
\[
N_{n,L} \hookrightarrow \widecheck{M}_{n,d}^{\mathrm{Dol}}, \quad \CE \mapsto (\CE, \theta=0),
\]
which induces 
\[
H^*( \widecheck{M}_{n,d}^{\mathrm{Dol}}, \BQ ) \rightarrow H^*(N_{n,L}, \BQ).
\]
Restricting to the $\Gamma$-invariant part further yields
\begin{equation}\label{res}
\mathrm{res}: H^*(\widehat{M}_{n,d}, \BQ) = H^*(\widecheck{M}_{n,d}, \BQ)^\Gamma \to H^*(N_{n,L}, \BQ).
\end{equation}
We use the same notation
\begin{equation}\label{taut_N}
\alpha_i \in H^{2i-2}(N_{n,L}, \BQ), ~~ \beta_i \in H^{2i}(N_{n,L}, \BQ),~~ \psi_{i,j} \in H^{2i-1}(N_{n,L}, \BQ), \quad 2\leq i\leq n,~~ 1\leq j \leq 2g
\end{equation}
to denote the tautological classes given by the image of the corresponding classes (\ref{PGL_taut}). Atiyah--Bott \cite{AB} showed that these tautological classes generate $H^*(N_{n,L}, \BQ)$ as a graded $\BQ$-algebra; equivalently, the map (\ref{res}) is surjective.

Let $\BA$ be the free graded-commutative $\BQ$-algebra 
\begin{equation}
    \label{eqn: A}
    \BA:= \BQ[\alpha_2,\dots, \alpha_n, \beta_2, \dots, \beta_n] \otimes \Lambda_{\BQ}(\psi_{2,j}, \dots, \psi_{n,j})_{1\leq j \leq 2g}.
\end{equation}
It can be viewed as a cohomology ring generated by the class (\ref{PGL_taut}), and all the relations are those forced by graded-commutativity. Theorem \ref{thm1.1} and the Atiyah--Bott tautological generation theorem \cite{AB} imply that the cohomology of $\widehat{M}_{n,d}$ and $N_{n,L}$ are both quotient rings of $\BA$; we write:
\[
H^*(\widehat{M}_{n,d}, \BQ) = \BA/I^M_{g,n}, \quad H^*(N_{n,L}, \BQ) = \BA/I^{\mathrm{EJK}}_{g,n,d}.
\]
Note that Theorem \ref{thm0.3} implies that the relations for $\widehat{M}_{n,d}$ are independent of $d$.

The ideal $I^{\mathrm{EJK}}_{g,n,d}$ (here $\mathrm{EJK}$ stands for Earl--Jeffrey--Kirwan) is well-understood. Indeed, the work of Jeffrey--Kirwan \cite{JK} gives an algorithm to determine the intersection theory of tautological classes on $N_{n,L}$; by the Poincar\'e duality this in principle gives all the relations among the classes (\ref{taut_N}). Later, Earl--Kirwan \cite{EK} provides a more explicit description of all the relations following an idea of Mumford. 

The ideal $I^M_{g,n}$ for the character variety $\widehat{M}_{n,d}$ is more mysterious. The only understood case is the ideal $I^M_{g,2}$ by Hausel--Thaddeus \cite{HT2}.  However, Theorem \ref{thm0.3} and the morphism (\ref{res}) yield strong constraints on $I^M_{g,n}$ for any genus $g\geq 2$ and rank $n\geq 3$, which we record in the following proposition.

\begin{prop}\label{prop3.11}
    Any relation $\CR \in I^M_{g,n}$ also lies in $I^{\mathrm{EJK}}_{g,n,d}$.
    Moreover, the ideal $I^M_{g,n}$ is homogeneous with respect to the Chern grading (\ref{Chern_grading}).
\end{prop}

An element $\CR \in I_{g,n,d}^{\mathrm{EJK}}$ is called a \emph{homogeneous EJK relation}, if it is homogeneous with respect to both the cohomological grading
\begin{equation}\label{coh_grading}
\mathrm{deg}(\alpha_i) = 2i-2, \quad \mathrm{deg}(\beta_i) = 2i, \quad \mathrm{deg}(\psi_{i,j}) = 2i-1,
\end{equation}
and the Chern grading
\[
\mathrm{deg}^C(\alpha_i) = \mathrm{deg}^C(\beta_i) = \mathrm{deg}^C(\psi_{i,j}) = i.
\]
By Proposition \ref{prop3.11}, any relation in $I^M_{g,n}$ has to be a homogeneous EJK relation. 

\begin{rmk}
For convenience, when we say that $\BA$ is a graded $\BQ$-algebra, we view it as a graded $\BQ$-algebra endowed with the cohomological grading (\ref{coh_grading}); we write
\[
\BA = \bigoplus_{i} \BA^i
\]
to be the corresponding decomposition. A graded automorphism 
\[
\varphi \in \mathrm{Aut}^{\mathrm{gr}}(\BA)
\]
is an automorphism which respects the cohomological grading. We use $\BA^i_{(j)}$ to denote the subspace of $\BA$ of cohomological degree $i$ and Chern degree $j$, \emph{i.e.},
\[
\BA = \bigoplus_{i,j} \BA^i_{(j)}.
\]
\end{rmk}

Since the Chern grading is realized as half the weight for the nonsingular character variety, the following lemma is a standard consequence of Hodge theory.

\begin{lem}\label{weight_constr}
If $\BA^i_{(j)} \neq 0$, then we have
\[
i \leq 2j \leq 2i.
\]
\end{lem}

The symplectic group $\mathrm{Sp}(H^1(C,\BZ))$ acts naturally on $\BA$ through the standard action on the odd generators $\psi_{i,j}$. We consider the $\operatorname{Sp}(H^1(C,\BZ))$-invariant part $\BA^{\operatorname{Sp}_{2g}}\subset \BA$. By the first fundamental theorem for the symplectic group, the $\mathrm{Sp}(H^1(C,\BZ))$-invariant subalgebra of the exterior algebra 
\[
\Lambda_{\BQ}(\psi_{2,j}, \dots, \psi_{n,j})_{1\leq j \leq 2g}
\]
is generated by 
\[
\gamma_{rs}
:=
\sum_{j=1}^{g}
\left(
\psi_{r,j}\psi_{s,j+g}
-
\psi_{r,j+g}\psi_{s,j}
\right),
\qquad 2\le r,s\le n.
\]
We conclude that the graded $\BQ$-algebra $\BA^{\mathrm{Sp}_{2g}}$ is generated by $\alpha_i,\beta_j, \gamma_{rs}$. Since we have already proven that the relation $\Upsilon_{g,n} \in I^M_{g,n}$ of minimal cohomological degree $r_{g,n}$ is unique, it has to be $\operatorname{Sp}(H^1(C,\BZ))$-invariant, \emph{i.e.},
\[
\Upsilon_{g,n} \in \BA^{\operatorname{Sp}_{2g}}.
\]

The following proposition completes the proof of Theorem \ref{thm0.8}(c).

\begin{prop}\label{prop3.12}
    For $g=2$ and $n=5$, there is a unique homogeneous $\operatorname{Sp}(H^1(C,\BZ))$-invariant EJK relation 
    \[
\Upsilon_{2,5} \in    I^{\mathrm{EJK}}_{2,5,1}
    \]
    in the cohomological degree $r_{2,5}=22$.
\end{prop}
\begin{proof}
We fix $L$ to be a line bundle on $C$ of degree 1, and let $N_{5,L}$ be the corresponding moduli space of stable vector bundles. Finding homogeneous EJK relations for $N_{5,L}$ is a linear algebra problem using the Jeffrey--Kirwan intersection pairing formula \cite[Theorem 9.12]{JK}. Due to the high computational complexity, the proof is completed using computer calculations. The codes and results are all posted in the public repository \cite{R5G2JK}. The \href{https://github.com/szqzs/R5G2JK/tree/main/docs}{documentation files} of the repository explain in more detail about the function of each part of the code and provide an explicit map connecting the math terms in \cite{JK} and the specific chunks of codes; we explain the repository briefly here.

The complex dimension of $N_{5,L}$ is $24$, and the Poincar\'e pairing on $N_{5,L}$ is $\mathrm{Sp}(H^1(C, \BQ))$-invariant. Therefore, the desired relations are given by elements $\gamma \in \mathbb{A}^{\operatorname{Sp}_4,22}$ satisfying that 
\[
\int_{N_{5,L}} \gamma \cup \gamma' = 0, \quad \forall \gamma' \in \BA^{\operatorname{Sp}_4, 26}.
\]
We note that, \emph{a priori}, such a $\gamma$ is not homogeneous with respect to the Chern grading.

Now, by Lemma \ref{weight_constr}, we decompose $\BA^{\operatorname{Sp}_4,22}$ as
\[
\BA^{\operatorname{Sp}_4,22} = \bigoplus_{c=11}^{22} \BA^{\operatorname{Sp}_4,22}_{(c)}
\]
according to the Chern grading. 
We fix a basis of $\BA$ given by the monomials of $\alpha_i, \beta_i, \psi_{i,j}$, and use this basis to express the pairing matrix \[
M_c= \left( \int_{N_{5,L}} \gamma \cup \gamma' \right)_{\gamma \in \BA^{\operatorname{Sp}_4,22}_{(c)},\; \gamma' \in \mathbb{A}^{\operatorname{Sp}_4,26}};
\]
it has size 
\[
\dim \BA^{\operatorname{Sp}_4,22}_{(c)} \times \dim \BA^{\operatorname{Sp}_4,26}.
\]
In principle, for each given $c \in [11, 22]$, we can use  \cite[Theorem 9.12]{JK} to calculate the full matrix $M_c$, calculate the left kernel of $M_c$, and verify that the kernel is 0 for $c\neq 12$ and that the left kernel is one-dimensional if $c=12$. 
That one-dimensional kernel then has to be the desired unique homogeneous relation in degree $22$.
However, this approach is too time-consuming, and we implemented the calculation in the following way so that the calculation can finish within a day on 4-8 CPUs.

The codes are mostly written in Python, which handles modular arithmetic better than rational arithmetic. Therefore, we try to do most calculations over a finite field.
Namely, we first choose a large prime $p=2^{61}-1$. For each $c\neq 12$, we calculate enough number of columns of $M_c \pmod p$ which is sufficient to conclude that $M_c(\mathbb{F}_p)$ is of full row rank. Therefore, the original rational matrix $M_c$ also has to be of full rank.
When $c=12$, we have $\dim_{\BQ} \BA^{\operatorname{Sp}_4,22}_{(12)}=44$.
We calculated the full matrix $M_c(\mathbb{F}_p)$ and we find that its row rank is 43. 
We then take a lift to $\mathbb{Z}$ of a generator of the left kernel of $M_c(\mathbb{F}_p)$, pair it with all the 1039 basis elements in $\mathbb{A}^{\operatorname{Sp}_4,26}$, and find that all the pairings are zero. Therefore, the lift has to be the desired $\Upsilon_{2,5}$.
\end{proof}

We note that the full (conjectural) HRV formula \cite{HRV} which calculates the weight polynomials for the character varieties predicts that the homogeneous EJK relation has to lie in $\BA^{22}_{(12)}$.  Our proof of Proposition \ref{prop3.12} does not rely on that; we show that when $c \neq 12$ there is no homogeneous EJK relation in $\BA^{22}_{(c)}$ at all.

\begin{rmk}\label{rmk_more}
    It is natural to search for higher $n$ for which the statement of Proposition \ref{prop3.12} holds. In view of Theorem \ref{thm0.5}, Conjecture \ref{main_conj} may also hold for those $n$'s.
\end{rmk}

We conclude this section by writing down the explicit form of $\Upsilon_{2,5}$, which is obtained from the computer calculation in the proof of Proposition \ref{prop3.12}.


\begin{thm}\label{thm: the unique relation}
For $d$ coprime to 5, the unique relation 
    \[
    \Upsilon_{2,5} = 0 \in H^{r_{2,5}}(\widehat{M}_{5,d}, \BQ) =  H^{22}(\widehat{M}_{5,d}, \BQ)
    \]
 is given by
    \begin{align*}
& \Upsilon_{2,5}:=130\alpha_2\beta_2^5
-3120\alpha_2\beta_2^3\beta_4
+6720\alpha_2\beta_2^2\beta_3^2
-180480\alpha_2\beta_2\beta_3\beta_5 +288\alpha_2\beta_2\beta_4^2
-31104\alpha_2\beta_3^2\beta_4\\
& \quad 
+576000\alpha_2\beta_5^2 +4480\alpha_3\beta_2^3\beta_3 
-90240\alpha_3\beta_2^2\beta_5
-62208\alpha_3\beta_2\beta_3\beta_4
+13824\alpha_3\beta_3^3
+1036800\alpha_3\beta_4\beta_5\\
&\quad 
 -780\alpha_4\beta_2^4
+288\alpha_4\beta_2^2\beta_4
-31104\alpha_4\beta_2\beta_3^2
+1036800\alpha_4\beta_3\beta_5
+440640\alpha_4\beta_4^2-90240\alpha_5\beta_2^2\beta_3\\
&\quad
+1152000\alpha_5\beta_2\beta_5
+1036800\alpha_5\beta_3\beta_4 -65\beta_2^4\gamma_{22}
+960\beta_2^3\gamma_{33}
-5120\beta_2^2\beta_3\gamma_{23}
+1560\beta_2^2\beta_4\gamma_{22}\\
& \quad
 -17280\beta_2^2\gamma_{35}
+4608\beta_2^2\gamma_{44}
+480\beta_2\beta_3^2\gamma_{22}
-19200\beta_2\beta_3\gamma_{25}
+13824\beta_2\beta_3\gamma_{34}
-18432\beta_2\beta_4\gamma_{24}\\
&\quad 
-14976\beta_2\beta_4\gamma_{33}+88320\beta_2\beta_5\gamma_{23}
+161280\beta_2\gamma_{55}
-20736\beta_3^2\gamma_{24}
-2304\beta_3^2\gamma_{33}+48384\beta_3\beta_4\gamma_{23}\\
&\quad 
+48000\beta_3\beta_5\gamma_{22}
+207360\beta_3\gamma_{45} +18288\beta_4^2\gamma_{22}
+69120\beta_4\gamma_{35}
-230400\beta_5\gamma_{25}
-345600\beta_5\gamma_{34}.
\end{align*}
\end{thm}

\section{Automorphism groups and proof of Theorem \ref{thm0.7}}\label{sec4}
In this section, we provide a numerical criterion (see Proposition \ref{normal_group}) to verify $G^\mathrm{wd}_{g,n} =G_{g,n}$; then we use this criterion and the relation $\Upsilon_{2,5}$ to prove Theorem \ref{thm0.7}. Combined with Theorem \ref{thm0.5}, we complete the proof of Theorem \ref{thm0.2}. We note that in this section, Proposition \ref{prop4.4} is the only place where the assumption $(g,n)=(2,5)$ is used. 

\subsection{A subgroup $H$ of $G_{g,n}$}

Since $R_{g,n} =  H^*(\widehat{M}_{n,d}, \BQ)$ is a finitely generated graded $\BQ$-algebra, we see that $G_{g,n}$ is a linear algebraic group over $\BQ$. We use $\mathfrak{g}_{g,n}$ to denote its Lie algebra, which underlies a finite-dimensional $\BQ$-vector space.

In \cite{dCMSZ}, a subgroup 
\[
H\simeq \mathrm{GSp}(H^1(C,\BQ))\times \BQ^* \subset G_{g,n}
\]
was described via the generators (\ref{generators}) of $R_{g,n}$. The group $H$ plays a crucial role for our purpose, which we review as follows.

Recall the symplectic basis $\{\delta_k\}_{k=1}^{2g}$ of $H^1(C, \BQ)$ in Section \ref{sec1} whose intersection matrix is 
\[
J_g:=\begin{pmatrix}
0 & \mathrm{Id}_g \\
-\mathrm{Id}_g & 0
\end{pmatrix}.
\]
Any element $\varrho\in \mathrm{GSp}(H^1(C, \BQ)) \subset \mathrm{GL}(H^1(C, \BQ))$ satisfies
\[
{^\mathfrak{t}}{\varrho} J_g \varrho = \chi_{\mathrm{sim}}(\varrho) J_g,
\]
with 
\begin{equation}\label{nu_char}
\chi_{\mathrm{sim}}: \mathrm{GSp}(H^1(C, \BQ)) \to \BQ^*
\end{equation}
the similitude character. It acts on the vertical vector ${^\mathfrak{t}}(\psi_{i,1}, \psi_{i,2}, \cdots, \psi_{i,2g})$ by left multiplication; we write this action on the $\BQ$-vector space $R^{2i-1}_{g,n}$  as
\[
\psi_{i,j} \mapsto \varrho(\psi_{i,j}).
\]
The action of any element
\[
\varphi : = (\varrho, t) \in  H\simeq  \mathrm{GSp}(H^1(C, \BQ)) \times \BQ^*, \quad \varrho \in \mathrm{GSp}(H^1(C, \BQ)),~~ t\in \BQ^*,
\]
on $R_{g,n}$ is given by
\begin{equation}\label{H_action}
\alpha_i \mapsto \chi_{\mathrm{sim}}(\varrho) t^i\alpha_i, \quad \beta_i \mapsto t^i \beta_i, \quad \psi_{i,j}\mapsto t^i \varrho(\psi_{i,j}).
\end{equation}
Geometrically, the group $H$ comes from combining the monodromy symmetries $\mathrm{Sp}(H^1(C, \BQ))$ and the action of $\BQ^* \times \BQ^*$ given by the bigrading of $R_{g,n}$.

The character lattice of $H$ is of rank 2,
\begin{equation}\label{additive}
\BX^*(H) = \BZ \chi_{\mathrm{sim}} \oplus \BZ \chi_{\mathrm{tor}},
\end{equation}
where $\chi_{\mathrm{sim}}$ is given by (\ref{nu_char}) and $\chi_{\mathrm{tor}}: \BQ^* \xrightarrow{\simeq} \BQ^*$ is the standard character of the torus $\BQ^*$. The $H$-action (\ref{H_action}) defines a natural $H$-action on the graded $\BQ$-algebra $\BA$ as in \eqref{eqn: A}.  Every $H$-fixed line $\BQ v \subset \BA$ corresponds to a character $\chi_v \in \BX^*(H)$ satisfying
\[
h v = \chi_v(h) \cdot v, \quad \forall h \in H.
\]
For convenience we use the additive notation (\ref{additive}) to present a character; for example, we have by (\ref{H_action}) that
\[
\chi_{\alpha_i} = \chi_{\mathrm{sim}} + i \chi_{\mathrm{tor}}, \quad \quad \chi_{\beta_i} = i\chi_{\mathrm{tor}}.
\]
Therefore the $H$-characters corresponding to the basis $\alpha^2_2, \alpha_3, \beta_2$ of $R^4_{g,n}$ are given by
\begin{equation}\label{H-char}
\chi_{\alpha^2_2} = 2\chi_{\mathrm{sim}} + 4 \chi_{\mathrm{tor}}, \quad \chi_{\alpha_3} = \chi_{\mathrm{sim}}+3 \chi_{\mathrm{tor}}, \quad \chi_{\beta_2} = 2 \chi_{\mathrm{tor}}.
\end{equation}
In view of Theorem \ref{thm0.8}(a,b), this $H$-action on $\BA$ has to preserve the line
\[
\BQ\Upsilon_{g,n} \subset \BA,
\]
which allows us to consider its associated character
\[
\chi_{\CR}  \in \BX^*(H). 
\]
The relation $\Upsilon_{g,n}$ is homogeneous with respect to both the cohomological grading $r_{g,n}$ and the Chern grading (which is defined as $c_{g,n}$). A direct calculation using (\ref{H_action}) shows that 
\[
\chi_{\CR} = \left(c_{g,n} -\frac{r_{g,n}}{2}\right) \chi_{\mathrm{sim}} +    c_{g,n} \chi_{\mathrm{tor}}.
\]

\begin{lem}\label{lem4.1}
    If $r_{g,n} \neq c_{g,n}$, then the two vectors $\chi_\CR$ and $\chi_{\alpha_2}$ span the rationalized character lattice $\BX^*(H)$.
\end{lem}

\begin{proof}
    It follows from the fact that $\dim_\BQ \BX^*(H) =2$ and the explicit formulas for $\chi_{\alpha_2}$ and $\chi_{\CR}$ above.
\end{proof}


The following proposition provides a criterion to verify the property (\ref{wd=G}) for the group $G_{g,n}$. As we will see in Section \ref{sec4.2}, the condition (\ref{linear-alge-final}) is linear-algebraic.

\begin{prop}\label{normal_group}
Assume that $r_{g,n} \neq c_{g,n}$. If $H$ is a normal subgroup of $G_{g,n}$, then (\ref{wd=G}) holds for the pair $(g,n)$. In particular, if 
    \begin{equation}\label{linear-alge-final}
    \dim_\BQ \mathfrak{g}_{g,n} \leq \dim_\BQ \mathfrak{gsp}(H^1(C, \BQ)) +1 (= 2g^2+g+2),
    \end{equation}
    then (\ref{wd=G}) holds for $(g,n)$.
\end{prop}

\begin{proof}
The second assertion follows immediately from the first. Indeed, if we have (\ref{linear-alge-final}), then the subgroup $H \subset G_{g,n}$ has to be equal to the identity component of the linear algebraic group $G_{g,n}$, which is clearly a normal subgroup. 

It suffices to prove the first assertion; we assume that $H$ is a normal subgroup of $G_{g,n}$. For a given element $\varphi \in G_{g,n}$, we consider the group homomorphism
\[
\tau_{\varphi}: H \to H, \quad h \mapsto \varphi^{-1} h \varphi,
\]
which is well-defined by the normality assumption. It induces a natural action on the rationalized character lattice
\[
\tau_{\varphi}^*: \BX^*(H)_\BQ \to \BX^*(H)_\BQ.
\]
We consider the $G_{g,n}$-action on $\BA$ induced by the $G_{g,n}$-action on $R_{g,n}$; such a lifting exists and is canonical since all the generators of $R_{g,n}$ have degrees $< r_{g,n}$ and there are no relations among them by Theorem \ref{thm0.8}(a,b). We note the following standard fact in group theory.

\medskip
\noindent {\bf Claim.} If a line $\BQ v \subset \BA$ is fixed by $G_{g,n}$, then we have 
\[
\tau^*_{\varphi} \; \chi_{v} = \chi_v, \quad \forall \varphi \in G_{g,n}.
\]

\begin{proof}[Proof of Claim]
By our assumption of $v$ and the definition of $\chi_v$, we have
\[
\varphi(v) =  \lambda v, \quad \lambda \in \BQ^*,~~ \varphi \in G_{g,n}, \quad \quad h(v) = \chi_v(h)v, \quad h\in H.
\]
Therefore the action of $h\varphi  \in G_{g,n}$ on the vector $v\in \BA$ satisfies
\begin{equation}\label{id1}
h\varphi  (v) = \lambda h(v) = \lambda \cdot \chi_v(h) \cdot v.
\end{equation}
On the other hand, using the identity $h\varphi(v) =\varphi(\tau_\varphi(h) (v))$, we obtain
\begin{equation}\label{id2}
h\varphi(v) = \varphi(\chi_v(\tau_\varphi(h)) v) = \lambda \cdot (\tau^*_\varphi\; \chi_v)(h) \cdot v.
\end{equation}
Combining (\ref{id1}) and (\ref{id2}), we conclude that for any $h\in H$ the following holds:
\[
\chi_v(h) =( \tau^*_\varphi \;\chi_v)(h).
\]
This proves the claim.
\end{proof}

Since the $G_{g,n}$-action on $\BA$ fixes both lines $\BQ\alpha_2$ and $\BQ\Upsilon_{g,n}$, by the claim above both characters
\[
\chi_{\alpha_2}, \chi_{\CR} \in \BX^*(H)_\BQ
\]
are fixed by $G_{g,n}$. Lemma \ref{lem4.1} and the assumption $r_{g,n} \neq c_{g,n}$ further imply that
\begin{equation}\label{phi=id}
\tau^*_\varphi = \mathrm{id}:\; \BX^*(H)_\BQ \to \BX^*(H)_\BQ.
\end{equation}
Since any element $\varphi \in G_{g,n}$ sends the weight space associated with a character $\chi \in \BX^*(H)$ to the weight space associated with $\tau^*_\varphi \;\chi \in \BX^*(H)$, we conclude from (\ref{phi=id}) that the action of $\varphi$ on $\BA$ has to preserve the weight space associated with each $H$-character.

This proves the desired statement, since the three lines 
\[
\BQ \alpha^2_2, ~~\BQ \alpha_3, ~~ \BQ \beta_2 \subset R^4_{g,n}
\]
are weight spaces associated with distinct $H$-characters (\ref{H-char}). 
\end{proof}

\subsection{Lie algebra calculations}\label{sec4.2}
The underlying $\BQ$-vector space associated with 
\[
\BA = \bigoplus_i \BA^i
\]
is infinite dimensional. For convenience, we consider the truncated $\BQ$-vector space
\[
\BA^{\leq} : = \bigoplus_{i \leq N} \BA^i.
\]
Here $N$ is any sufficiently large integer greater than $r_{g,n}$, so that all the generators $\alpha_i, \beta_i, \psi_{i,j}$ lie in $\BA^{\leq}$. Truncating the algebra structure on $\BA$ endows $\BA^{\leq}$ with a natural graded $\BQ$-algebra structure. Since all the generators of $\BA$ lie in $\BA^{\leq}$, we have
\begin{equation}\label{big_group}
\mathrm{Aut}^{\mathrm{gr}}(\BA) = \mathrm{Aut}^{\mathrm{gr}}(\BA^{\leq}).
\end{equation}
In the following, we replace $\BA$ by the truncated algebra $\BA^\leq$; this is only to ensure that everything we consider is obviously finite dimensional. For example, (\ref{big_group}) shows that $\mathrm{Aut}^{\mathrm{gr}}(\BA)$ is a finite dimensional linear algebraic group.

By Theorem \ref{thm0.8}(a,b), we see that $G_{g,n}$ is a subgroup of the stabilizer group 
\[
\mathrm{Stab}_{\mathrm{Aut}^{\mathrm{gr}}(\BA^{\leq})}(\BQ \Upsilon_{g,n})
\]
of the line $\BQ \Upsilon_{g,n} \subset \BA^{\leq}$ in $\mathrm{Aut}^{\mathrm{gr}}(\BA^{\leq})$. Therefore we have
\begin{equation}\label{upperbound}
\dim_\BQ \mathfrak{g}_{g,n} \leq \dim_\BQ \mathrm{Lie}\:\mathrm{Stab}_{\mathrm{Aut}^{\mathrm{gr}}(\BA^\leq)}(\BQ \Upsilon_{g,n}). 
\end{equation}
Calculating the dimension of the right-hand side of (\ref{upperbound}) is a simple linear algebra problem as we explain in the following.

We consider the $\BQ$-vector space $\mathrm{Der}^{\mathrm{gr}}(\BA^{\leq})$ of graded derivations on $\BA^\leq$, \emph{i.e.} $D\in \mathrm{Der}^{\mathrm{gr}}(\BA^{\leq})$ are derivations which preserve the grading:
\[
D(a_1 a_2) = a_1 D(a_2) +  D(a_1) a_2, \quad D: \BA^i \to \BA^i.
\]
The space $\mathrm{Der}^{\mathrm{gr}}(\BA^{\leq})$ is the Lie algebra of the group (\ref{big_group}). As a $\BQ$-vector space, it can be identified with
\[
\left(\bigoplus_{i=2}^{n} \BA^{2i-2} \right) \oplus \left(\bigoplus_{i=2}^{n} \BA^{2i} \right) \oplus \left(\bigoplus_{i=2}^{n}  \bigoplus_{j=1}^{2g} \BA^{2i-1}\right),
\]
where the first, the second, and the third factors parameterize the choices of $D(\alpha_i)$, $D(\beta_i)$, and $D(\psi_{i,j})$ respectively.

In order to calculate the tangent space of the stabilizer group of $\BQ\Upsilon_{g,n}$, we consider the \emph{linear map}:
\begin{equation}\label{linear2}
\rho_{\CR}: \mathrm{Der}^{\mathrm{gr}}(\BA^{\leq}) \to \BA^{r_{g,n}}/\BQ \Upsilon_{g,n}, \quad D \mapsto [D(\Upsilon_{g,n})].
\end{equation}
More precisely, this map is defined by applying the graded derivation $D$ to the element $\Upsilon_{g,n} \in \BA^{r_{g,n}}$ composed with the projection to the quotient space $\BA^{r_{g,n}}/\BQ\Upsilon_{g,n}$. 

\begin{lem}\label{lem4.33}
    We have
    \[
    \mathrm{Lie}\; \mathrm{Stab}_{\mathrm{Aut}^{\mathrm{gr}}(\BA^{\leq})}(\BQ \Upsilon_{g,n}) = \ker\; \rho_\CR.
    \]
\end{lem}

\begin{proof}\label{lem4.3}
    The condition $D\in \mathrm{Ker}(\rho_\CR)$ can be written explicitly as 
    \[
    D(\Upsilon_{g,n}) \in \BQ \Upsilon_{g,n};
    \]
    this is the infinitesimal version for the stabilization condition 
    \[
    \varphi(\Upsilon_{g,n}) \in \BQ \Upsilon_{g,n}. \qedhere
    \]
\end{proof}

The linear map (\ref{linear2}) is explicit in terms of the relation $\Upsilon_{g,n}$. The following proposition is deduced by a computer calculation from the formula of Theorem \ref{thm: the unique relation}.

\begin{prop}\label{prop4.4}
    For $(g,n)=(2,5)$, we have 
    \[
    \dim \ker (\rho_\CR) = 12.
    \]
\end{prop}

\begin{proof}
The proof is a direct matrix calculation using the explicit formula of $\Upsilon_{2,5}$ in Theorem \ref{thm: the unique relation}.
\href{https://github.com/szqzs/R5G2JK/blob/main/macaulay2/verify_stabilizer_lie_dimension.m2}
{The Macaulay2 script}
in the repository \cite{R5G2JK} constructs the free graded algebra $\BA$,
writes the relation $\Upsilon_{2,5}$, and lists all degree-preserving derivations.
There are $468$ such derivation parameters.
The script constructs the
quotient $\BA^{22}/\BQ\Upsilon_{2,5}$ explicitly by choosing one monomial appearing in $\Upsilon_{2,5}$ and using
the equation $\Upsilon_{2,5}=0$ to eliminate that monomial.  Since
$\dim_\BQ \BA^{22}=3868$,
this gives a quotient coordinate space of dimension $3867$.  Applying all
degree-preserving derivations to $\Upsilon_{2,5}$ gives a rational matrix of size
$3867\times 468$.  The script computes its rank over $\BQ$ to be $456$.
Therefore $\dim \ker(\rho_\CR)=468-456=12$.
\end{proof}

\begin{proof}[Finishing the proof of Theorem \ref{thm0.7}]
   Combining (\ref{upperbound}), Lemma \ref{lem4.33}, and Proposition \ref{prop4.4},  we obtain that 
   \[
   \dim_\BQ \mathfrak{g}_{2,5} \leq 12.
   \]
   Since $r_{2,5} = 22\neq 12 = c_{2,5}$, we conclude Theorem \ref{thm0.7} from Proposition \ref{normal_group}.
\end{proof}

\;

\appendix
\refstepcounter{section}
\section*{Appendix. Complex analytic degree dependence of Higgs moduli spaces}

In this appendix, we show that the Higgs moduli space for a fixed curve $C$ of genus $g \geq 2$, considered as a complex analytic variety, depends on the degree.

\begin{thm}\label{thm:fixed-curve-higgs-degree-dependence}
Let the genus $g(C)\geq 2$.
Let $d,d'\in \BZ$ satisfy $(n,d)=(n,d')=1$.
Then $\left(M^{\mathrm{Dol}}_{n,d}\right)^{\mathrm{an}}
\simeq
\left(M^{\mathrm{Dol}}_{n,d'}\right)^{\mathrm{an}}$
as complex analytic spaces if and only if $d'\equiv \pm d \pmod n$.
\end{thm}
\begin{proof}
If $d'\equiv \pm d \pmod n$, then the two moduli spaces are isomorphic as complex algebraic varieties as discussed in Section \ref{sec0.1}.

We now prove the only if part. 
The statement is vacuous if $n\leq 4$.
Therefore, we assume that $n\geq 5$ from now on, and assume there is an isomorphism of complex analytic spaces
\[
\Phi: \left(M^{\mathrm{Dol}}_{n,d}\right)^{\mathrm{an}}
\xrightarrow{\simeq}
\left(M^{\mathrm{Dol}}_{n,d'}\right)^{\mathrm{an}}.
\]

{\noindent \bf Step 1.} We first show that the Hitchin system is preserved.

\medskip

    By \cite[\S 2.2, Remark 2.8]{Survey}, the Hitchin morphism $h_d: M_{n,d}^{\mathrm{Dol}}\to A$ is the affinization morphism, \emph{i.e.}, the initial morphism from $M_{n,d}^{\mathrm{Dol}}$ to an affine scheme in the category of schemes.
    Note that $h_d$ is also the Remmert reduction, \emph{i.e.} the initial morphism from the analytification $(M_{n,d}^{\mathrm{Dol}})^\mathrm{an}$ to a Stein space in the category of complex analytic spaces. 
    Indeed, since $h_d$ is proper, the relative GAGA implies that analytification commutes with $h_{d,*}$ on coherent sheaves.
    Therefore, we have that \[
    h_{d,*}\CO_{M^{\mathrm{Dol}}_{n,d}}^{\mathrm{an}}\simeq \left( h_{d,*} \CO_{M_{n,d}^{\mathrm{Dol}}}  \right)^{\mathrm{an}} \simeq\CO_A^{\mathrm{an}}.
    \]
    Thus $h_{d}$ is the Remmert reduction by \cite[\S 2, Satz 1]{GrauertModifikationen}. By the uniqueness of Remmert reduction, we obtain that any biholomorphism $\Phi:(M^{\mathrm{Dol}}_{n,d})^{\mathrm{an}}
\xrightarrow{\simeq}
(M^{\mathrm{Dol}}_{n,d'})^{\mathrm{an}}$
induces a biholomorphism $\phi:A^{\mathrm{an}}\to A^{\mathrm{an}}$ such that $h_{d'}\circ \Phi=\phi\circ h_d$.

\medskip

{\noindent \bf Step 2.} We next argue that $\Phi$ identifies the nilpotent cones using the $\BG_m$-action.

\medskip

We know that every Hitchin morphism $h_d$ is $\BG_m$-equivariant with respect to the natural $\BG_m$-action given by scaling the Higgs field.
Now we want to discuss how canonical this $\BG_m$-equivariant structure is.
We take the natural $\BG_m$-equivariant structure on $M_{n,d}^{\mathrm{Dol}}$ using the modular description, and then endow $h_{d'}$ with the $\BG_m$-equivariant structure via taking conjugation by $\Phi$ and $\phi$.
\emph{A priori}, this endowed $\BG_m$-action may be very different than the natural modular action on $h_{d'}$.
However, we claim that they differ only by the rank 1 parts.
Namely, we must have 
\[
\phi(0)\in A_{1}:=H^0(C,\omega_C)\hookrightarrow \bigoplus_{i=1}^n H^0(C,\omega_C^{\otimes i})=A,
\]
where the inclusion is given by the $n$-th thickening of a degree 1 spectral curve 
\[
\omega\mapsto
\bigl(-n\omega,\binom n2\omega^2,\ldots,(-1)^n\omega^n\bigr).
\]
In the fixed-determinant trace-free case, this
is exactly \cite[Proposition 5.1]{BG-Higgs}. We explain in the following that similar ideas work in the $\operatorname{GL}_n$-case. 

We use $\BG_m^{d'}$ to denote the natural (algebraic) modular $\BG_m$-action on $A$, and let $\BG_m^{\phi}$ be the (holomorphic) $\BG_m$-action on $A$ induced by $\phi$. For any $a\in A$, the Hitchin fiber $h^{-1}_{d'}(a)$ deforms trivially along the orbit $\BG_m^{\phi}\cdot a$.
Now assume that $a$ corresponds to a nonsingular spectral curve $\pi_a: S_a\to C$.
Then we first claim that $S_a$ is not hyperelliptic. 
Indeed, we have $\omega_{S_a}\simeq\pi_a^*\omega_C^{\otimes n}$,
and that $\omega_C^{\otimes n}$ is very ample. If $S_a$ were
hyperelliptic, the morphism defined by the sections of the subspace $\pi_a^*H^0(C,\omega_C^{\otimes n})
\subset H^0(S_a,\omega_{S_a})$
would factor through the hyperelliptic map $S_a\to\BP^1$, forcing $\pi_a$ to factor through
$\BP^1$, which is impossible because the curve $C$ has genus $g\geq 2$. Since $S_a$ is not hyperelliptic and the Hitchin fiber over $a$ recovers the Jacobian variety of the curve $S_a$, the local Torelli theorem implies that $S_a$ itself also deforms trivially along the orbit $\BG_m^{\phi} \cdot a$. 

We consider the group $G:=H^0(C,\omega_C)\rtimes\BG_m^{d'}$ which
acts on the Hitchin base by translation of the spectral coordinate and
by the natural scaling action.
By \cite[Proposition 4.2]{BG-Higgs}, the kernel of the Kodaira-Spencer map $\ker\bigl(T_aA\rightarrow H^1(S_a,T_{S_a})\bigr)$ is identified with $T_a(G\cdot a)$.
Therefore, the orbit
$\BG_m^\phi\cdot a$ is contained in the single orbit $G\cdot a$; in particular, the boundary point $\phi(0) \in \overline{\BG^\phi_m\cdot a}$ lies in the orbit-closure $\overline{G\cdot a}$. We consider the quotient map
\[
q: A \to A':=A/A_1
\]
induced by the natural translation action of $A_1 = H^0(C, \omega_C)$ on the Hitchin base $A$. For any $a \in A$ corresponding to a nonsingular spectral curve $S_a$, we have 
\[
q(G\cdot a) = \BG_m\cdot q(a), \quad \overline{q(G\cdot a)} = \BG_m \cdot q(a) \cup \{0\}
\]
with $\BG_m$ the standard scaling action. Therefore, if  $q(\phi(0)) \neq 0 \in A'$, the image $q(\phi(0))$ must lie in the orbit $\BG_m\cdot q(a)$ for \emph{any} $a$ corresponding to a nonsingular spectral curve. This is not possible for dimension reasons since $a$ varies in a Zariski open subset of $A$. We have concluded that $q(\phi(0)) = 0\in A'$, \emph{i.e.}, the image $\phi(0)$ lies in the $A_1$-orbit of $0\in A$.

Since all the fibers of $h_{d'}$ over $A_1 \subset A$ are isomorphic to the nilpotent cone $h_{d'}^{-1}(0)$, we have a biholomorphism of nilpotent cones $h_{d}^{-1}(0)\simeq h_{d'}^{-1}(0)$. This must induce an algebraic isomorphism because the nilpotent cones are projective.

\medskip

{\noindent \bf Step 3.} Next, we show that the vector bundle component in the nilpotent cone is preserved.

\medskip

For the proof, we show that the bundle component is the unique irreducible component
of the $\mathrm{GL}_n$ nilpotent cone that does not admit a nontrivial
algebraic $\BG_m$-action. 
Indeed, on the one hand, the modular $\BG_m$-action is nontrivial for any component that is not $N^{\mathrm{Dol}}_{n,d}$. On the other hand, the determinant of every $\BG_m$-orbit of any $\BG_m$-action on $N^{\mathrm{Dol}}_{n,d}$ would give a morphism $\BG_m\rightarrow\operatorname{Pic}^d(C)$,
which must be constant. Therefore, any $\BG_m$-action on $N^{\mathrm{Dol}}_{n,d}$ must preserve the fixed determinant bundle moduli space. Finally, this fixed determinant bundle moduli space admits no nontrivial algebraic $\BG_m$-actions by
\cite[Proposition 3.1]{BG-Higgs}. 

In conclusion, the algebraic isomorphism $h_d^{-1}(0)\simeq h_{d'}^{-1}(0)$ must identify the unique components with this property, giving $N^{\mathrm{Dol}}_{n,d}\simeq N^{\mathrm{Dol}}_{n,d'}$.

\medskip

{\noindent \bf Step 4.} The theorem now follows from the degree-dependence result of
Harder--Narasimhan \cite[Corollary 3.3.4]{HN}.
\end{proof}

\begin{rmk}
Theorem \ref{thm:fixed-curve-higgs-degree-dependence} concerns the degree dependence of $M_{n,d}^{\mathrm{Dol}}$ as \textit{complex analytic spaces}, whereas the main body of this paper concerns the degree dependence of $M_{n,d}^{\mathrm{Dol}}$ as {homotopy types}. The latter degree-dependence phenomenon, when it occurs, is much stronger than the former and is, in general, much harder to prove. The proof of Theorem \ref{thm:fixed-curve-higgs-degree-dependence} uses the complex structure of $M_{n,d}^{\mathrm{Dol}}$ and the Hitchin system in an essential way; it is not obvious to us how to prove the corresponding statements for the complex analytic spaces associated with the de Rham or the Betti moduli spaces in general.
\end{rmk}


\begin{thebibliography}{99}



\bibitem{Alper} J. Alper, P. Belmans, D. Bragg, J. Liang, T. Tajakka, {\em Projectivity of the moduli space of vector bundles on a curve,} London Math. Soc. Lecture Note Ser., 480
Cambridge University Press, Cambridge, 2022, 90--125.




\bibitem{AB} M. Atiyah, R. Bott, {\em The Yang-Mills equations over Riemann surfaces,} Philos. Trans. Roy. Soc. London Ser. A 308 (1983), no. 1505, 523--615.



















\bibitem{PB} P. Bousseau, {\em Scattering diagrams, stability conditions, and coherent sheaves on $\BP^2$,} J. Algebr. Geom. 31(4), 593--686, 2022.




\bibitem{BG-Higgs}
I. Biswas and T. L. G\'omez,
{\em A Torelli theorem for the moduli space of Higgs bundles on a curve,}
Q. J. Math. 54 (2003), 159--169.




\bibitem{BGPG} S. Bradlow, O. Garc\'ia--Prada, P. Gothen, {\em Homotopy groups of moduli spaces of representations,} Topology 47 (2008), 203--224,








\bibitem{dCHM1} M. A. de Cataldo, T. Hausel, and L. Migliorini, {\em Topology of Hitchin systems and Hodge theory of character varieties: the case $A_1$,} Ann. of Math. (2) 175 (2012), no.~3, 1329--1407.

























\bibitem{dCMSZ} M. A. de Cataldo, D. Maulik, J. Shen, S. Zhang, {\em Cohomology of the moduli of Higgs bundles via positive
characteristic,} J. Eur. Math. Soc. 27 (2025), no. 4, 1385--1405.



\bibitem{dCZ}  M. A. de Cataldo, S. Zhang, {\em A cohomological non-abelian Hodge theorem in positive characteristic,} Algebr. Geom. 9 (2022), no. 5, 606--632.










\bibitem{Charles} F. Charles, {\em Conjugate varieties with distinct real cohomology algebras,} J. Reine Angew. Math. 630 (2009), pp. 125--139.


\bibitem{Chen-Zhu} T.-H. Chen, X. Zhu, Non-Abelian Hodge Theory for algebraic curves in characteristic p, Geom.
Funct. Anal. Vol. 25 (2015) 1706--1733.

















\bibitem{DU} G. D. Daskalopoulos and K. K. Uhlenbeck, {\em An application of transversality to the topology of the moduli space of stable bundles,} Topology 30 (1991), no. 4, 639--653.




\bibitem{Da} B. Davison, L. Hennecart, T. Kinjo, O. Schiffmann, and E. Vasserot, {\em Hecke operators on symplectic surfaces and $\chi$-independence,} arXiv 2607.01355.









\bibitem{EK} R. Earl, F. Kirwan, {\em Complete sets of relations in the cohomology rings of moduli spaces of holomorphic bundles and parabolic bundles over a Riemann surface,} Proc. London Math. Soc. (3) 89 (2004), no. 3, 570--622. 

























\bibitem{GrauertModifikationen}
H. Grauert,
{\em \"Uber Modifikationen und exzeptionelle analytische Mengen},
Math. Ann. 146 (1962), 331--368.


\bibitem{Groch} M. Groechenig, {\em Moduli of flat connections in positive characteristic,} Math. Res. Lett. 23 (2016), no.
4, 989--1047.


\bibitem{GS} M. Groechenig and S. Shen, {\em Complex $K$-theory of moduli spaces of Higgs bundles,} J. Eur. Math. Soc. (JEMS), published online, 2025.



\bibitem{GWZ} M. Groechenig, D. Wyss, and P. Ziegler, {\em Mirror symmetry for moduli spaces of Higgs bundles via p-adic integration,} Invent. Math. 221. 505--596(2020)









\bibitem{HN} G. Harder, M. S. Narasimhan, {\em On the cohomology groups of moduli spaces of vector bundles on curves,} Math. Ann. 212 (1974/75) 215--248.




\bibitem{H_Survey} T. Hausel, {\em Mirror symmetry and Langlands duality in the non-Abelian
Hodge theory of a curve,} ``Geometric Methods in Algebra and Number Theory'', Progress in Mathematics, Vol. 235 Fedor Bogomolov, Yuri
Tschinkel (Eds.), Birkh\"auser, 2005.





\bibitem{HRV} T. Hausel and F. Rodriguez-Villegas, {\em Mixed Hodge polynomials of character varieties,} with an appendix by Nicholas M. Katz, Invent. Math. 174 (2008), no. 3, 555--624. 





\bibitem{HT2} T. Hausel, M. Thaddeus, {\em Relations in the cohomology ring of the moduli space of rank 2 Higgs bundles,} J. Amer. Math. Soc. 16 (2003), no. 2, 303--327.

\bibitem{HT1} T. Hausel, M. Thaddeus, {\em Generators for the cohomology ring of the moduli space of rank 2 Higgs bundles,} Proc. London Math. Soc. (3) 88 (2004), no. 3, 632--658. 



\bibitem{Survey} T. Hausel, {\em Global topology of the Hitchin system,} Handbook of moduli. Vol. II, 29--69, Adv. Lect. Math. (ALM), 25, Int. Press, Somerville, MA, 2013.






\bibitem{HMMS} T. Hausel, A. Mellit, A. Minets, and O. Schiffmann, {$P = W$ via $\CH_2$,} arXiv:2209.05429.









\bibitem{Hit} N. J. Hitchin, {\em The self-duality equations on a Riemann surface,} Proc. London Math. Soc. (3) 55 (1987), no. 1, 59--126.

\bibitem{Hit1} N. J. Hitchin, {\em Stable bundles and integrable systems,} Duke Math. J. 54 (1987) 91--114.















\bibitem{JK} L. Jeffrey, F. C. Kirwan, {\em Intersection theory on moduli spaces of holomorphic bundles of arbitrary rank on a Riemann surface,} Ann. of Math. (2) 148 (1998), no. 1, 109--196.






















\bibitem{KK}
T. Kinjo and N. Koseki,
{\em Cohomological $\chi$-independence for Higgs bundles and Gopakumar--Vafa invariants,}
J. Eur. Math. Soc. (JEMS) 28 (2026), no. 2, 619--671.




\bibitem{LMP} W. Lim, M. Moreira, and W. Pi, {\em Cohomological $\chi$-dependence of ring structure for the moduli of one-dimensional sheaves on $\BP^2$,} Forum Math. Sigma. 12 (2024) No. e47, 17 pp.


























\bibitem{Markman} E. Markman, {\em Generators of the cohomology ring of moduli spaces of sheaves on symplectic surfaces,} J. Reine Angew. Math. 544 (2002), 61--82. 


\bibitem{Markman-integral} E. Markman, {\em Integral generators for the cohomology ring of moduli spaces of sheaves over Poisson surfaces}, Adv. Math. 208.2 (2007), 622--646.






\bibitem{MS_Pi} D. Maulik and J. Shen, {\em Endoscopic decompositions and the Hausel--Thaddeus conjecture,} Forum Math. Pi, 9, (2021), No. e8, 49 pp.

\bibitem{MS_chi} D. Maulik and J. Shen, {\em Cohomological $\chi$-independence for moduli of one-dimensional sheaves and moduli of Higgs bundles,} Geom. Topol. 27:4 (2023), 1539-1586. 

\bibitem{MS_PW} D. Maulik and J. Shen, {\em The $P=W$ conjecture for $\mathrm{GL}_n$,} Ann. of Math. (2) 200 (2024), no. 2, 529--556.

\bibitem{MSY2} D. Maulik, J. Shen, and Q. Yin, {\em  Algebraic cycles and Hitchin systems,} arXiv 2407.05188. To appear at Ann. Sci. \'Ec. Norm. Sup\'er,2026.









\bibitem{Mellit} A. Mellit, {\em Poincar\'e polynomials of moduli spaces of Higgs bundles and character varieties (no punctures),} Invent. Math.221, (2020), 301--327.
















\bibitem{NS} M. S. Narasimhan and C. S. Seshadri, {\em Stable and unitary vector bundles on a compact Riemann surface,} Ann of Math. Vol. 82, No. 3 (1965), pp. 540--567.










\bibitem{OV} A. Ogus and V. Vologodsky, {\em Nonabelian Hodge theory in characteristic $p$,}  Publ. Math. Inst. Hautes \'Etudes Sci. (2007), no. 106, 1–138. 





 







\bibitem{Ramanan} S. Ramanan, {\em The moduli spaces of vector bundles over an algebraic curve,} Math. Ann. 200, no. 1 (1973): 69-84.

\bibitem{Reed} D. Reed, {\em The topology of conjugate varieties,} Math. Ann. 305 (1996), no. 2, 287--309.



\bibitem{Sch} O. Schiffmann, {\em Indecomposable vector bundles and stable Higgs bundles over smooth projective curves,} Ann. of Math. (2), 183(1):297--362, 2016.




\bibitem{Serre}
J.-P. Serre, \emph{Exemples de vari\'et\'es conjugu\'ees non-hom\'eomorphes}, in {Oeuvres/Collected Papers, Vol.~II (1960--1971)}, pp.~246--249, Springer, 1986.







\bibitem{R5G2JK}
J. Shen and S. Zhang,
{\em Rank-5 genus-2 Jeffrey--Kirwan computation},
GitHub repository,
\url{https://github.com/szqzs/R5G2JK}.

\bibitem{Shende} V. Shende, {\em The weights of the tautological classes of character varieties,} Int. Math. Res. Not. 2017, no. 22, 6832--6840.


\bibitem{Sull} D. Sullivan, {\em Galois symmetry in manifold theory at the primes,} Proceedings of the ICM, Nice 1970, Vol. II, 169--175, 1971.









\bibitem{Simp} C. T. Simpson, {\em Higgs bundles and local systems,} Inst. Hautes \'Etudes Sci. Publ. Math. No. 75 (1992), 5--95.

\bibitem{Si1994II} C. T. Simpson, {\em Moduli of representations of the fundamental group of smooth projective varieties II,"} Publ. Math. IHES.  No. 80 (1994), 5--79.























\bibitem{Yu} H. Yu, {\em Comptage des syst\`em locaux $\ell$-adiques sur une courbe,} Ann. of Math. (2) 197 (2023), no. 2, 423--531.










\end{thebibliography}
\end{document}